\documentclass[a4paper,11pt]{article}
\usepackage[english]{babel}
\usepackage{amscd}
\usepackage{amsmath}
\usepackage{amsthm}
\usepackage{amssymb, faktor, titlesec}
\usepackage{graphicx, tikz-cd}
\usepackage[left=2.5cm,right=2.5cm, top=2.5cm,bottom=2.5cm,bindingoffset=0cm]{geometry}
\usepackage{enumerate, multicol, makecell}
\usepackage{setspace}
\usepackage{hyperref}

\usepackage{tikz}
\usetikzlibrary{positioning}
\usetikzlibrary{decorations.text}
\usetikzlibrary{decorations.pathmorphing}
\usetikzlibrary{decorations.markings}

\usetikzlibrary{arrows} 
\tikzset{
    >=stealth',
    pil/.style={
           ->,
           thick,
           shorten <=2pt,
           shorten >=2pt,}
}

\tikzset{->-/.style={decoration={
  markings,
  mark=at position .5 with {\arrow{>}}},postaction={decorate}}}
\tikzset{-<-/.style={decoration={
  markings,
  mark=at position .5 with {\arrow{<}}},postaction={decorate}}}  
\titleformat{\section}{\large\bfseries\filcenter}{\thesection}{1em}{}
\titleformat{\subsection}[runin]{\bfseries}{\thesubsection.}{1em}{}[]
\usepackage{enumitem}
\newlist{longenum}{enumerate}{5}
\setlist[longenum,1]{label=\roman*)}
\setlist[longenum,2]{label=\alph*)}

\usetikzlibrary{shapes.geometric}

\theoremstyle{plain}

\newtheorem{lemma}{Lemma}[section]
\newtheorem{proposition}[lemma]{Proposition}
\newtheorem{statement}[lemma]{Proposition}
\newtheorem{theorem}[lemma]{Theorem}
\newtheorem{corollary}[lemma]{Corollary}

\newtheorem{proposition-conjecture}[lemma]{Proposition-conjecture}

\theoremstyle{definition}
\newtheorem{definition}[lemma]{Definition}
\newtheorem{example}[lemma]{Example}
\newtheorem{remark}[lemma]{Remark}
\newtheorem{conjecture}[lemma]{Conjecture}
\newtheorem{problem}[lemma]{Problem}
\newcommand{\Ker}[1]{\mathrm{Ker} \, #1}

\newcommand{\codim}[1]{\mathrm{codim} \, #1}

\newcommand{\Z}{\mathbb{Z}}

\newcommand{\lt}{{B}}

\newcommand{\R}{\mathbb{R}}
\newcommand{\Complex}{\mathbb{C}}

\newcommand{\Id}{\mathrm{Id}}

\newcommand{\CP}{{\mathbb{C}}\mathrm{P}}

\newcommand{\Jac}{\mathrm{Jac}}
\newcommand{\Div}{\mathrm{Div}}
\newcommand{\BDiv}{\mathrm{InDeg}}

\newcommand{\GL}{\mathrm{GL}}
\newcommand{\PGL}{\mathbb{P}\mathrm{GL}}

\newcommand{\E}{\mathrm{Id}}

\newcommand{\gl}{\mathfrak{gl}}

\newcommand{\poly}{\pazocal P^\lt_{m,n}}
\newcommand{\glnc}{\mathrm{Mat}_n(\Complex)}

\newcommand{\indeg}{\mathrm{indeg}}
\newcommand{\Newton}{\mathrm{Newton}}
\newcommand{\diag}{\mathrm{diag}}
\newcommand{\mult}{\mathrm{mult}}

\newcommand{\modspace}[3]{\pazocal M_{#3}}
\newcommand{\Str}{\pazocal S}
\newcommand{\eqdef}{\mathrel{\mathop:}=}
\renewcommand{\deg}{\mathrm{deg}\,}
\usepackage{calrsfs}
\DeclareMathAlphabet{\pazocal}{OMS}{zplm}{m}{n}

\title{Matrix polynomials, generalized Jacobians, \\and graphical zonotopes}
\author{ Anton Izosimov\footnote{Department of Mathematics,
University of Toronto, Toronto, ON M5S 2E4, Canada; e-mail: izosimov@math.utoronto.ca}
}
\date{ }
\begin{document}
\maketitle
\begin{abstract}
A {matrix polynomial} is a polynomial in a complex variable $\lambda$ with coefficients in $n \times n$ complex matrices. The spectral curve of a matrix polynomial $P(\lambda)$ is the curve $\{ (\lambda, \mu) \in \Complex^2 \mid \det(P(\lambda) - \mu \cdot \E) = 0\}$. The set of matrix polynomials with a given spectral curve $C$ is known to be closely related to the Jacobian of $C$, provided that $C$ is smooth. We extend this result to the case when $C$ is an arbitrary nodal, possibly reducible, curve. In the latter case the set of matrix polynomials with spectral curve $C$ turns out to be naturally stratified into smooth pieces, each one being an open subset in a certain generalized Jacobian. We give a description of this stratification in terms of purely combinatorial data and describe the adjacency of strata. We also make a conjecture on the relation between completely reducible matrix polynomials and the canonical compactified Jacobian defined by V.\,Alexeev.

% that are biholomorphic to open subsets in certain generalized Jacobians.

%It is well known that if the spectral curve $C$ is smooth, then the set of matrix polynomials isospectral with $P(\lambda)$ is intimately related to the Jacobian of the spectral curve. The aim of the present paper is to study the relation between isospectral matrix polynomials and the Jacobian of the spectral curve in the case when the spectral curve is nodal and possibly reducible.

% It is well known that if $C$ is a smooth plane algebraic curve, then the set of matrix polynomials whose spectral curve is $C$ is closely related to the Jacobian variety of $C$. We extend this relation to the case when $C$ is a nodal, possibly reducible, curve.
\end{abstract}
\tableofcontents
\section{Introduction}
The study of isospectral matrix polynomials lies at the crossroads of integrable systems and algebraic geometry. An $n \times n$ \textit{matrix polynomial} of degree $m$ is an expression of the form $P(\lambda)= A_0 + A_1\lambda + \dots + A_m \lambda^m$ where $A_0, \dots, A_m \in \mathrm{Mat}_n(\Complex)$ are complex $n \times n$ matrices, and $\lambda$ is a complex variable. It is well known that there is a close relation between matrix polynomials and line bundles over plane curves. With a given matrix polynomial $P(\lambda)$, one can associate a plane algebraic curve $\{ (\lambda, \mu) \in \Complex^2 \mid \det(P(\lambda) - \mu \cdot \E) = 0\}$, called the \textit{spectral curve},  and a line bundle over the spectral curve given by eigenvectors of $P(\lambda)$. Conversely, given a plane algebraic curve $C$ and a line bundle $E \to C$ satisfying certain genericity assumptions, one can construct a matrix polynomial whose spectral spectral curve is $C$ and whose eigenvector bundle is isomorphic to $E$.\par

%
% By definition, a point $(\lambda, \mu) \in \Complex^2$ belongs to the spectral curve $C$ if $\mu$ is the eigenvalue of the matrix $P(\lambda)$. The fiber of the bundle $E$ over the point $(\lambda, \mu) $ is defined to be the corresponding eigenvector. Because of this correspondence, there is a close relation between matrix polynomials and moduli spaces of line bundles. In particular, when the spectral curve $C$ is smooth, the set of matrix polynomials isospectral with the given one is closely related to the Jacobian variety of $C$.
%
\par
%matrix polynomials are closely related to plane algebraic curves. 

%The \textit{spectral curve} associated with a matrix polynomial $P(\lambda)$ is an algebraic curve defined as the zero locus of the characteristic polynomial  $\det(P(\lambda) - \mu \Id)$ in the $\lambda$-$\mu$ plane. 
%We are interested in studying the sets of isospectral matrix polynomials, i.e. those pencils whose spectral curve is the given curve $C$.

%\par\bigskip
%Two matrix polynomials are called \textit{isospectral} if they have the same characteristic polynomials, or, equivalently, the same spectral curves. It is well known that the structure of the set of isospectral matrix pencils to a large extent reflects the complex geometry of the corresponding spectral curve. 
When the spectral curve is smooth, the relation between matrix polynomials and line bundles is described by the following classical result. Fix positive integers $m$ and $n$, and let $\modspace{m}{n}{C}$ be the moduli space of $n \times n$ degree $m$ matrix polynomials with spectral curve $C$, considered up to conjugation by constant matrices (for simplicity of notation, we omit the dependency of $\pazocal M_C$ on $m$ and $n$; note that the dimension $n$ can always be found from the equation of the spectral curve $C$).
Then, if the curve $C$ is smooth and satisfies certain conditions at infinity, the moduli space $\modspace{m}{n}{C}$ is a smooth variety biholomorphic to the complement of the theta divisor in the Jacobian of~$C$. For hyperelliptic curves, this result is due to D.\,Mumford. In~\cite{Mumford}, Mumford used the space $\modspace{m}{n}{C}$ to give an explicit algebraic construction of the hyperelliptic Jacobian. For general smooth curves $C$, the space $\modspace{m}{n}{C}$ was studied by P.\,Van Moerbeke and D.\,Mumford \nolinebreak\cite{mvm} in connection with periodic difference operators, and by M.\,Adler and P.\,Van Moerbeke \nolinebreak\cite{adler}, A.G.\,Reyman and M.A.\,Semenov-Tian-Shansky~\cite{Reyman2}, and A.\,Beauville \nolinebreak\cite{beauville2} in the context of finite-dimensional integrable systems. It is worth mentioning that the technique used in these papers is to a large extent based on earlier works of S.P.\,Novikov's school on infinite-dimensional integrable systems, see, e.g., the review \cite{DKN}.
%, and works of S.P.\,Novikov's school on algebro-geometric solutions of soliton equations~\cite{DKN}.
%Different versions of this classical result belong to D.\,Mumford\,\cite{Mumford}, P.\,Van Moerbeke and D.\,Mumford \nolinebreak\cite{mvm}, M.\,Adler and P.\,Van Moerbeke \nolinebreak\cite{adler}, A.G.\,Reyman and M.A.\,Semenov-Tian-Shansky~\cite{Reyman2}, and A.\,Beauville \nolinebreak\cite{beauville2}. %Namely, it is shown in \cite{beauville2} that if the curve $p(\lambda, \mu) = 0$ is smooth, then the moduli space of polynomial matrices whose characteristic polynomial is $p$, considered up to conjugation by constant matrices, is the Jacobian of this curve minus the theta divisor. 
%and works of the Soviet school on algebro-geometric solutions of soliton equations~\cite{DMN}).\par
\par
% In view of this result, the variety $\modspace{m}{n}{C}$ associated with a smooth curve $C$ may be regarded as a model for the affine part of the Jacobian of $C$. 
In the present paper, we are interested in generalizing the relation between isospectral matrix polynomials and Jacobian varieties to singular curves. It is known that for certain singular curves the space $\modspace{m}{n}{C}$ is a variety isomorphic to the complement of the theta divisor in the compactification of the generalized Jacobian of $C$, see e.g. \cite{inoue}. However, in general, the space $\modspace{m}{n}{C}$ associated with a singular curve may even fail to be Hausdorff. For this reason, we replace the space $\modspace{m}{n}{C}$ with a slightly bigger space $\pazocal P^B_C$ defined as the set of $n \times n$ degree $m$ matrix polynomials whose spectral curve is a given curve $C$ and whose leading coefficient $A_m$ is equal to a fixed matrix $B$. 
 %Nevertheless, there always exists an extension of $\modspace{m}{n}{C}$  which is a variety. This extension $\pazocal P^B_C$ is the space of degree $m$ and dimension $n$ matrix polynomials with spectral curve $C$ and fixed leading coefficient $A_m = B$. 
 %For this reason, instead of $\modspace{m}{n}{C}$, we consider a slightly bigger space $\pazocal P^B_C$ of degree $m$ and dimension $n$ matrix polynomials with spectral curve $C$ and fixed leading coefficient $A_m = B$. 
 The latter space is, by definition, an affine algebraic variety for any curve $C$ and any matrix $B$ (by an algebraic variety we always mean an algebraic set, i.e. we allow varieties to be reducible). Note that the space $\modspace{m}{n}{C}$ can be recovered from  $\pazocal P^{B}_C$ by taking the quotient of the latter with respect to the conjugation action of the centralizer of $B$, provided that this quotient is well-defined.\par
For a smooth curve $C$, the set $\pazocal P^B_C$ was explicitly described by L.\,Gavrilov \cite{Gavrilov2}. Namely, he showed that if $C$ is smooth, then $\pazocal P^B_C$ can be naturally identified with an open dense subset in the generalized Jacobian of the curve obtained from $C$ by identifying points at infinity. We note that the latter Jacobian has a natural fiber bundle structure over the Jacobian of $C$, which agrees with the above description of the moduli space $\modspace{m}{n}{C}$ as the quotient of $\pazocal P^B_C$.\par
In this present paper, we consider the case of an arbitrary nodal and possibly reducible curve~$C$.
%In the case of a nodal reducible curve $C$,  the variety $\pazocal P^B_C$ also turns out to be reducible. We show that the problem of enumeration of these comp components and describing their intersections boils down to a certainf 
In this case, the isospectral set $\pazocal P^B_C$ turns out to have a rich combinatorial structure. This structure may be described in terms of orientations on the dual graph of $C$, or in terms of a certain convex polytope, the so-called {graphical zonotope} \cite{postnikov} associated with the dual graph. Based on this combinatorial description, we make a conjecture on the relation of the set $\pazocal P^B_C$ to the canonical compactified Jacobian of $C$ described by V.\,Alexeev \cite{Alexeev}.\par

The interest in varieties $\pazocal P^B_C$ associated with singular curves also comes from integrable systems. %Isospectral matrix polynomial play a significant role in both finite and infinite-dimensional integrable systems. 
There is a natural algebraically integrable system on matrix polynomials of the form $ A_0 + L_1 \lambda + \dots + A_{m-1}\lambda^{m-1} + B\lambda^m$ whose fibers are precisely the sets $\pazocal P^B_C$, see~\cite{Reyman2, beauville2, Gavrilov2}. So, the results of the present paper can be regarded as a description of singular fibers for the integrable system on matrix polynomials. %The application of our results to integrable systems will be discussed in more detail in a separate publication.\par
We also hope that these results will be useful on their own, in particular, for understanding the structure of compactified Jacobians.\par
 Main results of the paper are Theorem \ref{thm1} and Theorem \ref{thm3}. Theorem \ref{thm1} describes the variety $\pazocal P^B_C$ associated with a nodal curve $C$ as a disjoint union of smooth strata, each one being an open subset in a certain generalized Jacobian. Theorem \ref{thm3} describes the adjacency of these strata. Two other results are Theorem \ref{thm2}, which describes the local structure of the variety $\pazocal P^B_C$, and Theorem \ref{thm4}, which characterizes strata  of $\pazocal P^B_C$ corresponding to irreducible and completely reducible matrix polynomials. Based on Theorem \ref{thm4}, we make a conjecture on the relation of the variety $\pazocal P^B_C$ to the canonical compactified Jacobian of $C$ (Conjecture \ref{conjCCJ}).\par
 The main emphasis of the present paper is on constructions and examples. Detailed proofs will be published elsewhere. \par
%The combinatorics of this stratification is described in terms of so-called indegree divisors,
%\begin{example}
%The following example outlines the basic constructions below. Let $C$ be the union of $n$ straight lines $C_i = \{ (\lambda, \mu) \in \Complex^2 \mid \mu = a_i + \lambda b_i\}$, and let $\pazocal P^B_C$ be the set of matrix polynomials of the form $A_0 + \lambda B$ with spectral curve $C$ where $A_0,B \in \glnc$, and $B$ is a fixed diagonal matrix with diagonal entries $b_1, \dots, b_n$. 
%\end{example}
\bigskip

{\bf Acknowledgments.}
This work was partially supported by the Dynasty Foundation Scholarship and an NSERC research grant.

\section{Statement of the problem}\label{mainProblem}
Fix positive integers $m,n$, and let
%%$P_m(\gl_n, J)$ of $\gl_n(\Complex)$-valued polynomials of degree $d$ with a fixed leading term $B$:
\begin{equation*}\pazocal P_{m,n} \eqdef \left\{  P(\lambda) = A_0 + A_1 \lambda + \dots + A_{m-1}\lambda^{m-1} + A_m\lambda^m \mid A_0, \dots, A_{m} \in \glnc \right\} %\subset \glnc \otimes \Complex[\lambda]
\end{equation*} 
be the space of $n \times n$ degree $m$ matrix polynomials. For each $P \in \pazocal P_{m,n}$, let
 \begin{align*}
%% \label{specDef} 
C_P\eqdef  \{ (\lambda, \mu) \in \Complex^2 \mid\det(P(\lambda)-\mu \cdot \E) = 0\}\end{align*}
 be the zero set of its characteristic polynomial. The curve $C_P$ is called the \textit{spectral curve} associated with a matrix polynomial $P$. This curve can be regarded as the graph of the spectrum of $P$ depending on the value of $\lambda$. Further, fix a matrix $B \in \glnc$ with distinct eigenvalues, and let
$$
\poly = \{ P(\lambda) = A_0 + A_1 \lambda + \dots + A_{m-1}\lambda^{m-1} + A_m\lambda^m \in \pazocal P_{m,n} \mid A_m = B \}
$$
be the space of matrix polynomials $P \in \pazocal P_{m,n}$ with leading coefficient $B$.
 The isospectral variety $\pazocal P^B_C$ corresponding to a plane curve $C$ is defined as
%The curve $C_P$ encodes the spectrum of $P(\lambda)$ for every $\lambda \in \Complex$ and is called the \textit{spectral curve} associated with $P$.\par 
%Consider the following {inverse spectral problem}: \textit{given an affine algebraic curve $C$, what is the structure of the set
\begin{align*}
\pazocal P^B_C \eqdef   \{P \in \poly \mid  C_P = C \}.
\end{align*}
%of matrix polynomials with the spectral curve $C$?} In what follows, the set  $\isv_C$ will be referred to as the \textit{isospectral set} corresponding to the curve $C$. Note that $\isv_C$ is an algebraic set (i.e. the zero set of a finite number of polynomials), but possibly reducible. So, $\isv_C$ is, in general, not a variety (recall that a variety is, by definition, an irreducible algebraic set).
  The main problem of the present paper is the following:
  \begin{problem}\label{problem1}
   Describe the structure of the set $\pazocal P^B_C$ for an arbitrary nodal, possibly reducible, curve $C$. 
   \end{problem}
 % \begin{remark}
 The solution of Problem \ref{problem1} is given by Theorems \ref{thm1}, \ref{thm2}, and \ref{thm3} below. We begin with several preliminary remarks.
 First, note that if the matrices $B$ and $B'$ are similar, then the corresponding isospectral varieties $\pazocal P^{B}_C$ and $\pazocal P^{B'}_C$ are naturally isomorphic. For this reason, we may confine ourselves to the case of a diagonal matrix $B$ (recall that eigenvalues of $B$ are pairwise distinct).\par

Further, note that the choice of $m$, $n$, and $B$ imposes restrictions on curves which may arise as spectral curves. Indeed, let $C$ be a curve given by $Q(\lambda, \mu) = 0$. Assume that  $\pazocal P^B_C $ is non-empty, and let $$P(\lambda) = A_0 + A_1 \lambda + \dots + A_{m-1}\lambda^{m-1} + B\lambda^m \in  \pazocal P^B_C. $$ Introduce new variables $z,w$ such that $\lambda = 1/z$, and $\mu = w/z^m$. Then
\begin{align}
\begin{aligned}\label{specCond}\lim_{z \to 0}\left(Q \left({\frac{1}{z}}, \frac{w}{z^{m}}\right)z^{nm} \right) =  \lim_{z \to 0}\left(\det\left(P - \frac{w}{z^m}\cdot\Id\right)z^{mn}\right) = \\
\lim_{z \to 0}\vphantom{\int}\det(B + zA_{m-1} + \dots + z^m A_0 - w\cdot\E) &=
 \det(B- w\cdot\E).
 \end{aligned}
 \end{align}
Geometrically, this condition means that the Newton polygon of the polynomial $Q(\lambda, \mu)$ lies inside the triangle with vertices $(0,0)$, $(0,n)$, $(mn, 0)$, and that coefficients of the monomials lying on the hypothenuse of this triangle are equal to corresponding coefficients of the characteristic polynomial of the matrix $B$ (see Figure~\ref{newton}). In what follows, we consider only those curves $C$ which satisfy this condition. Note that for smooth and nodal curves, the above condition turns out to be not only necessary, but also a sufficient condition for a curve $C$ to be the spectral curve for a suitable matrix polynomial $P$. Also note that for matrix polynomials of degree one, the above condition simply means that the curve $C$ has degree $n$, and that its closure in \nolinebreak $\CP^2$ intersects the line at infinity at specific points prescribed by the eigenvalues of the matrix $B$. %For  $m \geq 2$, condition  \eqref{specCond} is much more restrictive. %In particular, the defining polynomial  $p$ should contain only those monomials $\lambda^i \mu^j$ which satisfy $i + mj \leq mn$.
% In particular, for $2 \times 2$ matrices, the spectral curve $C$ must be hyperelliptic.
%\end{remark}
    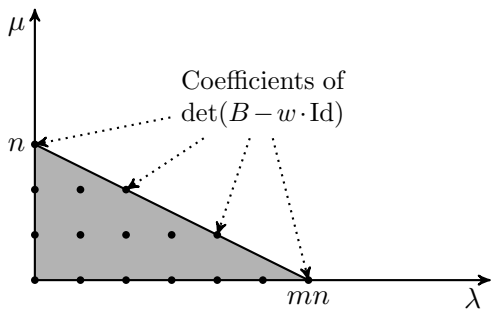
\begin{figure}[t]
\centerline{
\begin{tikzpicture}[thick,scale = 1.2]
   		 \draw  [->] (0,0) --(0,3);
		  \draw  [->] (0,0) --(5,0);
		  \draw (3,0) -- (0,1.5);
		  \fill[opacity = 0.3] (0,0) -- (3,0) -- (0,1.5) -- cycle;
		  \fill (0,0) circle (1.2pt);
		    \fill (0,0.5) circle (1.2pt);
		      \fill (0,1) circle (1.2pt);
		        \fill (0,1.5) circle (1.2pt);
		    %  {  \node[draw, fill, star, star points=5, star point ratio=2.25] at (0,1.5){};}
		        		  \fill (0.5,0) circle (1.2pt);
		    \fill (0.5,0.5) circle (1.2pt);
		      \fill (0.5,1) circle (1.2pt);		       
		      		        		  \fill (1,0) circle (1.2pt);
		    \fill (1,0.5) circle (1.2pt);
		      \fill (1,1) circle (1.2pt);
		      		        		  \fill (1.5,0) circle (1.2pt);
		    \fill (1.5,0.5) circle (1.2pt);
		      	 		      		        		  \fill (2,0) circle (1.2pt);
		    \fill (2,0.5) circle (1.2pt);
		    		      		        		  \fill (2.5,0) circle (1.2pt);
		    \fill (3,0) circle (1.2pt);
%		    \node () at (-0.2, 0.5) {$1$};
%		        \node () at (-0.2, 1) {$2$};
		              \node () at (-0.2, 1.5) {$n$};
%		              		    \node () at (0.5, -0.2) {$1$};
%				    		    \node () at (1, -0.2) {$2$};
%						    		    \node () at (1.5, -0.2) {$3$};
%								    		    \node () at (2, -0.2) {$4$};
%										    		    \node () at (2.5, -0.2) {$5$};
												    		    \node () at (3, -0.2) {$mn$};
														    		    \node () at (-0.2, 2.8) {$\mu$};
																    \node () at (4.8, -0.2) {$\lambda$};   
																    \node (A) at (2.5,2) {\small \parbox{2.15cm}{Coefficients of $\det(B - w \cdot \E)$}};
																    \draw [->, dotted] (A) -- (0,1.5);
																       \draw [->, dotted] (A) -- (1,1);
																          \draw [->, dotted] (A) -- (2,0.5);
																             \draw [->, dotted] (A) -- (3,0);
\end{tikzpicture}
}
\caption{Newton polygon of a spectral curve}\label{newton}
\end{figure}

Finally, we note that there is a natural action of the group   \begin{align*}G^\lt \eqdef \{ U \in \PGL_n(\Complex) \mid U\lt = \lt U \}\end{align*} on $\pazocal P^B_C$ by conjugation. Provided that $\pazocal P^B_C$ is non-empty, the quotient $\pazocal P^B_C / G^B$ coincides with the moduli space $$\modspace{m}{n}{C} = \{P \in \pazocal P_{m,n} \mid C_P = C  \} \,/\,\GL_n(\Complex),$$
where the group $\GL_n(\Complex)$ acts by conjugation. Indeed, there is a natural map
$
\pi \colon \pazocal P^B_C \to \modspace{m}{n}{C}
$
that sends every $P \in \pazocal P^B_C$ to its conjugacy class. Obviously, fibers of this map are precisely $G^B$ orbits. Furthermore, since $ \pazocal P^B_C$ is non-empty, there \textit{exists} $P \in \modspace{m}{n}{C}$ such that its leading coefficient is conjugated to $B$. But this implies that this is so  \textit{for any} $P \in \modspace{m}{n}{C}$, because the eigenvalues of the leading coefficient of a matrix polynomial are uniquely determined by the spectral curve. Therefore, the map $\pi$ is surjective, and we have  $$\modspace{m}{n}{C} = \pazocal P^B_C / G^B. $$

\begin{example}\label{twoLines}
The following example shows that if the curve $C$ is reducible, then the space $\modspace{m}{n}{C}$ is, generally speaking, not a variety. Let $m =1$, $n=2$, and let $C$ be the union of two straight lines $\mu = a_1 + b_1 \lambda$ and $\mu = a_2 + b_2 \lambda$. To describe the space $\modspace{m}{n}{C}$, we first describe the variety $\pazocal P^B_C$ for $B$ diagonal with diagonal entries $b_1, b_2$. A simple explicit computation shows that $\pazocal P^B_C$ is a disjoint union of three $G^B$ orbits $\mathrm{Orb}_1 \sqcup \mathrm{Orb}_2 \sqcup \mathrm{Orb}_3$ where
\begin{align*}
&\qquad\qquad\qquad\qquad \mathrm{Orb}_1 = \left\{ \left(\begin{array}{cc}a_1 +  b_1 \lambda & 0 \\0 & a_2 +  b_2 \lambda \end{array}\right)  \right\},
\\  &\mathrm{Orb}_2 =  \left\{ \left(\begin{array}{cc}a_1 + b_1\lambda  &z \in \Complex^* \\0 & a_2 + b_2 \lambda \end{array}\right)  \right\}, \quad  \mathrm{Orb}_3 = \left\{ \left(\begin{array}{cc}a_1 + b_1 \lambda & 0 \\z \in \Complex^* & a_2 + b_2 \lambda\end{array}\right)   \right\}.
\end{align*}
Note that the closure of the orbit $\mathrm{Orb}_2$, as well as the closure of the orbit $\mathrm{Orb}_3$, contains the orbit $\mathrm{Orb}_1$, therefore the quotient $\modspace{m}{n}{C} = \pazocal P^B_C / G^B $ is non-Hausdorff.
\end{example}

%  \begin{remark}
%  The sets $\pazocal P^B_C$ can be interpreted in terms of integrable systems. There is a natural Poisson structure on the set $\poly$, the so-called \textit{linear $r$-matrix bracket}. The coefficients of the characteristic polynomial $\det(P(\lambda)-\mu \cdot \E)$, considered as polynomials on $\poly$, commute with respect to this Poisson structure and form a completely integrable system \cite{Reyman2}. The Lagrangian fibers of this system are precisely the sets  $\pazocal P^B_C$. %\par Note that the sets $\modspace{m}{n}{C}$ have a similar interpretation in terms of so-called \textit{Mumford systems}~\cite{Mumford, inoue} and their generalizations \cite{pedroni}.
%  
%  \end{remark}
  
  \section{Stratification of the isospectral variety}\label{genCons}
  Let $C$ be an arbitrary nodal, possibly reducible, curve. In this section, with each matrix polynomial $P \in  \pazocal P^B_C$, we associate a graph $\Gamma_P$ and a divisor $D_P$ on $\Gamma_P$, i.e. an integer-valued function on vertices of $\Gamma_P$. This construction leads to a stratification of  $\pazocal P^B_C$ with strata indexed by pairs (a graph $\Gamma$, a divisor $D$ on $\Gamma$) .\par  
   \tikzstyle{vertex} = [coordinate]
% \begin{figure}[t]
% \centering
%\begin{tikzpicture}[thick]
%\node () at (-1.7,1.2) {\includegraphics[scale = 1.2]{curve.pdf}};
%\node [vertex] at (2,0.5) (A) {};
%\node [vertex] at (2,1.45) (B) {};
%\node at (2,0.5) (Ad) {};
%\node at (2,1.45) (Bd) {};
%    \draw  (A) -- (B);
%        \draw  (A) .. controls +(-0.3,+0.5) .. (B);
%        \draw   (A) .. controls +(+0.3,+0.5) .. (B);
%    \fill (A) circle [radius=1.5pt];
%    \fill (B) circle [radius=1.5pt];
%       \draw  (A) .. controls +(-0.5,-0.7) and +(+0.5,-0.7) .. (A);
%       \draw[->, dashed] (-1,0.5) -- (Ad);
%           \draw[->, dashed] (0.2,1.45) -- (Bd);
%    \node at (2.5,1.7) () {$\Gamma_C$};      
%        \node at (-3,2) () {$C$};     
%\end{tikzpicture}
%\caption{A nodal curve and its dual graph}\label{curveGraph}
%\end{figure}
 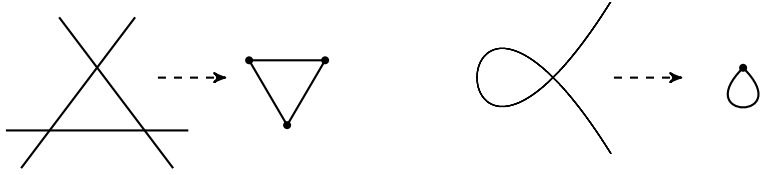
\begin{figure}[t]
 \centering
\begin{tikzpicture}[thick]
\draw (0,0) -- (1.5,2);
\draw (2,0) -- (0.5,2);
\draw (-0.2,0.5) -- (2.2,0.5);
\draw[->, dashed] (1.8,1.2) -- (2.7,1.2);
\node () at (3.5,1) {
\begin{tikzpicture}[thick, rotate = 180]
 \draw   (0,0) -- (0.5, 0.86);
    \fill (0,0) circle [radius=1.5pt];
       \fill (1,0) circle [radius=1.5pt];
          \fill (0.5,0.86) circle [radius=1.5pt];
   		 \draw  (0,0) -- (1,0);    
 		   \draw    (0.5, 0.86) -- (1,0); 
		   \end{tikzpicture}
};
\node () at (7,1.2) {\includegraphics[scale = 1]{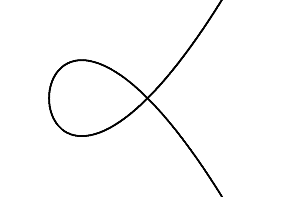}};
\draw[->, dashed] (7.8,1.2) -- (8.7,1.2);
\node () at (9.5,1) {
\begin{tikzpicture}[thick]
% \draw   (0,0) -- (0.5, 0.86);
    \fill (0,0) circle [radius=1.5pt];
    \draw (0,0) .. controls +(-0.7,-0.7) and +(+0.7,-0.7) .. (0,0);
		   \end{tikzpicture}
};
\end{tikzpicture}
\caption{Nodal curves and their dual graphs}\label{curveGraph}
\end{figure}

Recall that \textit{dual graph} $\Gamma_C$ of a nodal curve $C$ is defined as follows. The vertices of $\Gamma_C$ are irreducible components of $C$. The edges of $\Gamma_C$ are nodes. Two vertices are joined by an edge if there is a node joining the corresponding irreducible components. Nodes which lie in only one irreducible component correspond to loops in $\Gamma_C$ (see e.g. Figure~\ref{curveGraph}).\par
  
  %Condition \eqref{specCond} implies that $C$ has $n$ branches at infinity and can be compactified by adding $n$ smooth points $\infty_1, \dots, \infty_n$. Such a compactification may be constructed by embedding $C$ into a weighted projective space, see e.g. \cite{Gavrilov2}. For $m=1$, we simply consider the closure of $C$ in $\CP^2$.\par
  
  Now we associate a subgraph $\Gamma_P \subset \Gamma_C$ with each matrix polynomial in $P^B_C$. Let $P \in \pazocal P^B_C$. Then, by definition of the spectral curve, for any point $ (\lambda, \mu) \in C$ we have $\dim \Ker(P(\lambda) - \mu \cdot \Id) > 0$. Moreover, if $ (\lambda, \mu)$ is a smooth point of $C$, then $\dim \Ker(P(\lambda) - \mu \cdot \Id) = 1$, i.e. the $\mu$-eigenspace of the matrix $P(\lambda)$ is one-dimensional (see e.g.~\cite{babelon2003introduction}, Chapter~5.2). If, on the contrary, $ (\lambda, \mu) \in C$ is a double point, then there are two possibilities: either $\dim \Ker(P(\lambda) - \mu \cdot \Id) = 1$, or $\dim \Ker(P(\lambda) - \mu \cdot \Id) = 2$ (cf. Section~2 of~\cite{adams1990}). %This leads to a subdivision of the variety  $\pazocal P^B_C$ into strata. Each stratum is singled out by fixing the dimension of $\Ker(P(\lambda) - \mu \cdot \Id) $ at each double point  $ (\lambda, \mu) \in C$.
  Now, define $\Gamma_P \subset \Gamma_C$ as the graph obtained from the dual graph $\Gamma_C$ by removing all edges of  $\Gamma_C$ corresponding to nodes $(\lambda, \mu) \in C$ such that $\dim \Ker(P(\lambda) - \mu \cdot \Id) = 2$. Note that the subgraph $\Gamma_P \subset \Gamma_C$ is \textit{generating}, i.e. it has the same set of vertices as the graph $\Gamma_C$ itself.
 % Let us denote by $Sing(C)$ the set of double points of $C$. For each subset $K \subset $
  
  \begin{example}\label{twoLines2}
Let $m =1$, $n=2$, and let the curve $C$ be the union of two lines $C_1 = \{\mu = a_1 + b_1 \lambda\}$ and $C_2 =  \{\mu = a_2 + b_2 \lambda\}$, as in Example \ref{twoLines}. Let also $B$ be a diagonal matrix with diagonal entries $b_1$, $b_2$. Then we have $\pazocal P^B_C  = \mathrm{Orb}_1 \sqcup \mathrm{Orb}_2 \sqcup \mathrm{Orb}_3$ (see Example \ref{twoLines}). The only double point of the curve $C$ is $(\lambda_0, \mu_0) = C_1 \cap C_2$, and the dual graph $\Gamma_C$ is two vertices joined by an edge. For $P \in \mathrm{Orb}_1$, we have $\dim \Ker(P(\lambda_0) - \mu_0 \cdot \Id) = 2$, so the corresponding graph $\Gamma_P$ is two disjoint vertices. For $P \in \mathrm{Orb}_2$ or $P \in \mathrm{Orb}_3$, we have $\dim \Ker(P(\lambda_0) - \mu_0 \cdot \Id) = 1$, so $\Gamma_P = \Gamma_C$.   \end{example}
  As can be seen from Example \ref{twoLines2}, different components of  $\pazocal P^B_C$ may correspond to the same generating subgraph $\Gamma_P \subset \Gamma_C$. To distinguish between these components, with each matrix polynomial  $P \in \pazocal P^B_C$ we associate an additional invariant, a divisor $D_P$ on the graph $\Gamma_P$. This divisor is defined as follows.\par
  
 Let $C_1, \dots, C_k$ be the irreducible components of $C$. Consider the normalization $X$ of $C$, i.e. the disjoint union of connected Riemann surfaces $X_1, \dots, X_k$, where $X_i$ is the normalization of $C_i$. Note that $X$ is equipped with two meromorphic functions $\lambda$ and $\mu$ coming from the embedding of the curve $C$ into $\Complex^2$.  For all but finitely many points $u \in X$, we have $\dim \Ker(P(\lambda(u)) - \mu(u) \cdot \Id) = 1 $. This allows us to construct a densely defined map $\psi_P \colon X \to \CP^{n-1}$ which sends $u$ to the one-dimensional subspace $ \Ker(P(\lambda(u)) - \mu(u) \cdot \Id) \subset \Complex^n$. Since $X$ is smooth, the map $\psi_P$ extends to a holomorphic map defined on the whole $X$.
%  
%  
%  us resolve those double points of $C$ where $\dim \Ker(P(\lambda) - \mu \cdot \Id) = 2$ (cf. Section 2 of \cite{adams1990}). As a result, we obtain a nodal curve $C'_P$ whose dual graph is exactly the graph $\Gamma_P$ constructed above.  The curve $C'_P$ has the following property: there exists a unique regular map $\psi_P \colon C'_P \to \CP^{n-1}$ such that for any point $R \in C'_P$, we have $\psi_P(R) \subset \Ker(L(\lambda(R)) - \mu(R) \cdot \Id)$. Moreover, $C'_P$ is a minimal partial normalization of $C$ with this property.\par
  Now, consider a line bundle over $X$ defined by taking the line $\psi_P(u)$ for every point $u \in X$. Denote this line bundle by $E_P$, and let $E_P^*$ be the dual bundle. The total degree of the bundle $E_P^*$ can be found in the same way as in the smooth case (see e.g. \cite{babelon2003introduction}, Chapter 5.2): \begin{align}
  \label{totalDegreeFormula}
  \sum_{i=1}^k \deg E_P^*\mid_{X_i} \,= \sum_{i=1}^k \,(\mathrm{genus}(X_i) - 1) +  |E(\Gamma_P)| + n \end{align} where $|E(\Gamma_P)|$ stands for the number of edges of $\Gamma_P$, or, equivalently, the number of nodes $(\lambda, \mu) \in C$ such that $\dim \Ker(P(\lambda) - \mu \cdot \Id) = 1$. For smooth curves, this formula reduces to the standard one 
\begin{align}\label{smoothDegree}
   \deg E_P^*= \mathrm{genus}(X)  + n -1.
   \end{align}
%The sum on the left-hand side of~\eqref{totalDegreeFormula} is uniquely determined by the graph $\Gamma_P$, but individual summands are not.
  
  %However, it turns out that the degree of the restriction of the bundle $E_P^*$ to each component $X_i$

    \begin{example}\label{twoLines3}
    The following example shows that while the {total} degree of the bundle $E_P^*$ can be computed in terms of the curve $C$ and the graph $\Gamma_P$, the degree of the restriction of $E_P^*$ to $X_i$ can be different for different matrix polynomials $P$ with the same spectral curve and the same graph. \par Let us again consider the two-lines spectral curve from Examples \ref{twoLines}, \ref{twoLines2}. The corresponding Riemann surface is $X = X_1 \sqcup X_2$ where $X_i$ is the closure of the line $C_i = \{\mu = a_i + b_i\lambda\}$ in $\CP^2$. Take a matrix polynomial $P \in \mathrm{Orb}_2$. Then for any point $u = (\lambda, \mu) \in C_1$, we have
    $$
    P - \mu \cdot \Id = \left(\begin{array}{cc}0 & z \\0 & (a_2 + b_2\lambda) - (a_1 + b_1\lambda)\end{array}\right),    $$
so $  \Ker(P - \mu \cdot \Id)$ is spanned by the vector $w =  (1,0)$, which implies that the restriction of the bundle $E_P$ to $X_1$ is a trivial bundle, and $\deg E_P^*\mid_{X_1} = 0$. Further, for $u = (\lambda, \mu) \in C_2$, we have
    $$
    P - \mu \cdot \Id = \left(\begin{array}{cc}(a_1 + b_1\lambda) - (a_2 + b_2\lambda)\ & z \\0 & 0\end{array}\right),    $$
    so $  \Ker(P - \mu \cdot \Id)$ is spanned by the vector $w = (z, (a_2 + b_2\lambda) - (a_1 + b_1\lambda))$. The vector-function $w$ can be regarded as a meromorphic section of the bundle $E_P$ over $X_2$ with one pole of order one at infinity, so $\deg E_P\mid_{X_2} = -1$, and $\deg E_P^*\mid_{X_2} = 1$. 
   Analogously, for $P \in \mathrm{Orb}_3$, we get $\deg E_P^*\mid_{X_1} = 1$, and $\deg E_P^*\mid_{X_2} = 0$. Note that in both cases  $P \in \mathrm{Orb}_2$ and  $P \in \mathrm{Orb}_3$, the total degree of the line bundle $ E_P^*$ is equal to one, as predicted by formula \eqref{totalDegreeFormula}.    \end{example}
   Now, we use the numbers $\deg E_P^*\mid_{X_i} $ to define a divisor $D_P$ on the graph $\Gamma_P$. Take any vertex $v_i$ of $\Gamma_P$. By definition, this vertex corresponds to an irreducible component $C_i$ of $C$ or, equivalently, to a connected component $X_i \subset X$. Define
           \begin{align}\label{dpdef}
     D_P(v_i):= \deg E^*_P\mid_{X_i}\!-\, (\mathrm{genus}(X_i) + n_i - 1)
         \end{align}
   where $n_i$ is the number of poles of the function $\lambda$ on $X_i$ or, equivalently, the degree of the equation of the irreducible component $C_i$ in the variable $\mu$. Note that for smooth curves the right-hand side of \eqref{dpdef} is zero (cf. formula \eqref{smoothDegree}), so the divisor $D_P$ measures the deviation of the degree of the line bundle $E^*_P\mid_{X_i}$ from the degree that it would have in the smooth case. 

 \begin{figure}[t]
\centerline{
\begin{tikzpicture}[thick,scale = 1.2]
 \node (a) at (0,0)
         {
            \begin{tikzpicture}[scale = 1.2]
            \node () at (-4,0) {$\left(\begin{array}{cc}a_1 +  b_1 \lambda & 0 \\0 & a_2 +  b_2 \lambda \end{array}\right)$};
            \fill (0,0) circle (1.2pt);
            \fill (1,0) circle (1.2pt);
   		 \draw  [dashed] (0,0) -- (1,0);    
		\node () at (-0.1,-0.2) {$0$};
		\node () at (1.1,-0.2) {$0$};	
				%\node () at (0,0.5) {(a) $\mathrm{Orb}_1$};	
				\draw [->, dashed] (-1.75,0) -- (-0.75,0);
            \end{tikzpicture}
         }; 
          \node (b) at (0,-1.5)
         {
            \begin{tikzpicture}[scale = 1.2]
            \node () at (-4,0) {$\left(\begin{array}{cc}a_1 +  b_1 \lambda & * \\0 & a_2 +  b_2 \lambda \end{array}\right)$};
            \fill (0,0) circle (1.2pt);
            \fill (1,0) circle (1.2pt);
   		 \draw  [] (0,0) -- (1,0);    
		\node () at (-0.1,-0.2) {$0$};
		\node () at (1.1,-0.2) {$1$};	
				%\node () at (0,0.5) {(a) $\mathrm{Orb}_1$};	
				\draw [->, dashed] (-1.75,0) -- (-0.75,0);
            \end{tikzpicture}
         }; 
                   \node (b) at (0,-3)
         {
            \begin{tikzpicture}[scale = 1.2]
            \node () at (-4,0) {$\left(\begin{array}{cc}a_1 +  b_1 \lambda & 0 \\ * & a_2 +  b_2 \lambda \end{array}\right)$};
            \fill (0,0) circle (1.2pt);
            \fill (1,0) circle (1.2pt);
   		 \draw  [] (0,0) -- (1,0);    
		\node () at (-0.1,-0.2) {$1$};
		\node () at (1.1,-0.2) {$0$};	
				%\node () at (0,0.5) {(a) $\mathrm{Orb}_1$};	
				\draw [->, dashed] (-1.75,0) -- (-0.75,0);
            \end{tikzpicture}
         }; 
%          \node (b) at (3,0)
%         {
%            \begin{tikzpicture}[scale = 1.2]
%            \fill (0,0) circle (1.2pt);
%            \fill (1,0) circle (1.2pt);
%   		 \draw  [] (0,0) -- (1,0);    
%		\node () at (-0.1,-0.2) {$0$};
%		\node () at (1.1,-0.2) {$1$};	
%			\node () at (0,0.5) {(b) $\mathrm{Orb}_2$};
%            \end{tikzpicture}
%         }; 
%                   \node (c) at (6,0)
%         {
%            \begin{tikzpicture}[scale = 1.2]
%            \fill (0,0) circle (1.2pt);
%            \fill (1,0) circle (1.2pt);
%   		 \draw  [] (0,0) -- (1,0);    
%		\node () at (-0.1,-0.2) {$1$};
%		\node () at (1.1,-0.2) {$0$};	
%			\node () at (0,0.5) {(c) $\mathrm{Orb}_3$};
%            \end{tikzpicture}
%         }; 
\end{tikzpicture}
}
\caption{Graphs and divisors for matrix polynomials whose spectral curve is two straight lines}\label{diangle}
\end{figure}
  \begin{example}\label{twoLines4}
  Figure \ref{diangle} depicts the graph $\Gamma_P$ and the divisor $D_P$ for each of the components $\mathrm{Orb}_1$, $\mathrm{Orb}_2$,  and $\mathrm{Orb}_3$ of  the variety $\pazocal P^B_C$ considered in Examples \ref{twoLines}, \ref{twoLines2}, \ref{twoLines3}. Here and in what follows, dashed edges are those which do not belong to the subgraph $\Gamma_P$.
  \end{example}
  %In what follows, we characterize divisors $D$ for such that $D = D_L$ for an appropriate $P \in \pazocal P^B_C$. 
  
  %  
%  By definition, $$d_L := \sum_{v_i} \left(\deg E^*_P \mid_{C_i}\right)v_i$$ where the sum is taken over all vertices of $\Delta_L$, and $C_i$ is an irreducible component of $C'_P$ corresponding to the vertex $v_i$. 

Now, for any generating subgraph $\Gamma$ of the dual graph of $C$, and for any divisor $D$ on $\Gamma$, define the corresponding stratum
  \begin{align*}
  \Str(\Gamma,D) := \{ P \in  \pazocal P^B_C \mid \Gamma_P = \Gamma, D_P = D\}.
\end{align*}
%We will also use the notation $\Str(\mbox{picture})$:
  \begin{example}\label{twoLines4_5}
  When $C$ is two lines (see Example \ref{twoLines} and Examples \ref{twoLines2} - \ref{twoLines4}), the variety $\pazocal P^B_C$ consists of three strata
  \begin{align*}\Str\left(
     \begin{tikzpicture}[baseline={([yshift=-.5ex]current bounding box.center)},vertex/.style={anchor=base,
    circle,fill=black!25,minimum size=18pt,inner sep=2pt}]
            \fill (0,0) circle (1.2pt);
            \fill (1,0) circle (1.2pt);
   		 \draw  [dashed] (0,0) -- (1,0);    
		\node () at (-0.2,-0.) {$0$};
		\node () at (1.2,-0.) {$0$};	
            \end{tikzpicture}
  \right) = \mathrm{Orb}_1, \quad
\Str\left(
     \begin{tikzpicture}[baseline={([yshift=-.5ex]current bounding box.center)},vertex/.style={anchor=base,
    circle,fill=black!25,minimum size=18pt,inner sep=2pt}]
            \fill (0,0) circle (1.2pt);
            \fill (1,0) circle (1.2pt);
   		 \draw  [] (0,0) -- (1,0);    
		\node () at (-0.2,-0.) {$0$};
		\node () at (1.2,-0.) {$1$};	
            \end{tikzpicture}
  \right) = \mathrm{Orb}_2, \quad \Str\left(
     \begin{tikzpicture}[baseline={([yshift=-.5ex]current bounding box.center)},vertex/.style={anchor=base,
    circle,fill=black!25,minimum size=18pt,inner sep=2pt}]
            \fill (0,0) circle (1.2pt);
            \fill (1,0) circle (1.2pt);
   		 \draw  [] (0,0) -- (1,0);    
		\node () at (-0.2,-0.) {$1$};
		\node () at (1.2,-0.) {$0$};	
            \end{tikzpicture}
  \right) =  \mathrm{Orb}_3.  \end{align*}
  \end{example}

Theorem \ref{thm1} below asserts that non-empty strata $\Str(\Gamma ,D)$ are smooth and connected, and, moreover, can be described as open subsets in semiabelian varieties. This result generalizes the description of $\pazocal P^B_C $ for smooth curves $C$ found by Gavrilov \cite{Gavrilov2}. 
   The rest of the paper is devoted to the description of those divisors $D$ for which $\Str(\Gamma ,D)$ is non-empty. We also describe the local structure of the variety $\pazocal P^B_C $ in the neighborhood of each stratum and characterize those strata which are adjacent to each other.

  \section{Combinatorial interlude I: Indegree divisors and graphical zonotopes}
   In what follows, a \textit{graph} is an undirected multigraph possibly with loops. The notation $V(\Gamma)$ stands for the vertex set of a graph $\Gamma$, and $E(\Gamma)$ stands for the edge set. By $\Div(\Gamma)$ we denote the set of all divisors on $\Gamma$. Each divisor $D \in \Div(\Gamma)$ may be regarded either as a formal integral linear combination of vertices of $\Gamma$, or as a $\Z$-valued function on vertices. The \textit{degree} $|D|$ of a divisor $D$ is the sum of its values over all vertices. %We denote the set of all divisors on a graph $\Gamma$ by $\Div(\Gamma)$.
 \par
  Let $\Gamma$ be a graph, and let $\pazocal O(\Gamma)$ be the set of all orientations on $\Gamma$. For each $\mathfrak{o} \in \pazocal O$, define a divisor \begin{align}\label{indegDiv}\indeg(\mathfrak o) :=\!\!\! \sum_{v\, \in\, V(\Gamma)} \!\indeg_{\mathfrak o}(v)v\end{align}
  where $\indeg_{\mathfrak o}(v)$ is the indegree of the vertex $v$ in the directed graph $(\Gamma, \mathfrak o)$, i.e. the number of edges pointing to $v$.
  \begin{definition}\label{balDef1}
  Let $\Gamma$ be a graph. We say that a divisor $D \in \Div(\Gamma)$ is an \textit{indegree divisor} if there exists an orientation $\mathfrak o \in \pazocal O(\Gamma)$ such that $D = \indeg(\mathfrak o)$.  We denote the set of all indegree divisors by $\BDiv(\Gamma)$.%In other words, a divisor $d$ is balanced if there exists an orientation of $\Gamma$ such that for every vertex $v \in \Gamma$, the number of edges entering $v$ is equal to $d(v_i)$.
    \end{definition}

 \begin{figure}[b]
\centerline{
\begin{tikzpicture}[thick,scale = 1.2]
 \node (a) at (0,0)
         {
            \begin{tikzpicture}[scale = 1.2]
   		 \draw  [->-] (0,0) -- (0.5, 0.86);
   		 \draw  [->-] (0,0) -- (1,0);    
 		   \draw  [->-]  (0.5, 0.86) -- (1,0); 
		\node () at (-0.1,-0.1) {$0$};
		\node () at (1.1,-0.1) {$2$};	
				\node () at (0.5,1.02) {$1$};	
            \end{tikzpicture}
         }; 
          \node (b) at (2,0)
         {
            \begin{tikzpicture}[scale = 1.2]
   		 \draw  [->-] (0,0) -- (0.5, 0.86);
   		 \draw  [->-] (0,0) -- (1,0);    
 		   \draw  [-<-]  (0.5, 0.86) -- (1,0); 
		\node () at (-0.1,-0.1) {$0$};
		\node () at (1.1,-0.1) {$1$};	
				\node () at (0.5,1.02) {$2$};	
            \end{tikzpicture}
         }; 
                   \node (c) at (4,0)
         {
            \begin{tikzpicture}[scale = 1.2]
   		 \draw  [->-] (0,0) -- (0.5, 0.86);
   		 \draw  [-<-] (0,0) -- (1,0);    
 		   \draw  [-<-]  (0.5, 0.86) -- (1,0); 
		\node () at (-0.1,-0.1) {$1$};
		\node () at (1.1,-0.1) {$0$};	
				\node () at (0.5,1.02) {$2$};	
            \end{tikzpicture}
         }; 
                            \node (d) at (6,0)
         {
            \begin{tikzpicture}[scale = 1.2]
   		 \draw  [-<-] (0,0) -- (0.5, 0.86);
   		 \draw  [-<-] (0,0) -- (1,0);    
 		   \draw  [-<-]  (0.5, 0.86) -- (1,0); 
		\node () at (-0.1,-0.1) {$2$};
		\node () at (1.1,-0.1) {$0$};	
				\node () at (0.5,1.02) {$1$};				
            \end{tikzpicture}
         }; 
                            \node (e) at (8,0)
         {
            \begin{tikzpicture}[scale = 1.2]
   		 \draw  [-<-] (0,0) -- (0.5, 0.86);
   		 \draw  [-<-] (0,0) -- (1,0);    
 		   \draw  [->-]  (0.5, 0.86) -- (1,0); 
		\node () at (-0.1,-0.1) {$2$};
		\node () at (1.1,-0.1) {$1$};	
				\node () at (0.5,1.02) {$0$};	
            \end{tikzpicture}
         }; 
                            \node (f) at (10,0)
         {
            \begin{tikzpicture}[scale = 1.2]
   		 \draw  [-<-] (0,0) -- (0.5, 0.86);
   		 \draw  [->-] (0,0) -- (1,0);    
 		   \draw  [->-]  (0.5, 0.86) -- (1,0); 
		\node () at (-0.1,-0.1) {$1$};
		\node () at (1.1,-0.1) {$2$};	
				\node () at (0.5,1.02) {$0$};	
            \end{tikzpicture}
         }; 
                            \node (g) at (4,-2)
         {
            \begin{tikzpicture}[scale = 1.2]
   		 \draw  [-<-] (0,0) -- (0.5, 0.86);
   		 \draw  [->-] (0,0) -- (1,0);    
 		   \draw  [-<-]  (0.5, 0.86) -- (1,0); 
		\node () at (-0.1,-0.1) {$1$};
		\node () at (1.1,-0.1) {$1$};	
				\node () at (0.5,1.02) {$1$};	
            \end{tikzpicture}
         }; 
                            \node (h) at (6,-2)
         {
            \begin{tikzpicture}[scale = 1.2]
   		 \draw  [->-] (0,0) -- (0.5, 0.86);
   		 \draw  [-<-] (0,0) -- (1,0);    
 		   \draw  [->-]  (0.5, 0.86) -- (1,0); 
		\node () at (-0.1,-0.1) {$1$};
		\node () at (1.1,-0.1) {$1$};	
				\node () at (0.5,1.02) {$1$};	
            \end{tikzpicture}
         }; 
\end{tikzpicture}
}
\caption{Eight orientations and seven indegree divisors on a triangle}\label{triangle}
\end{figure}
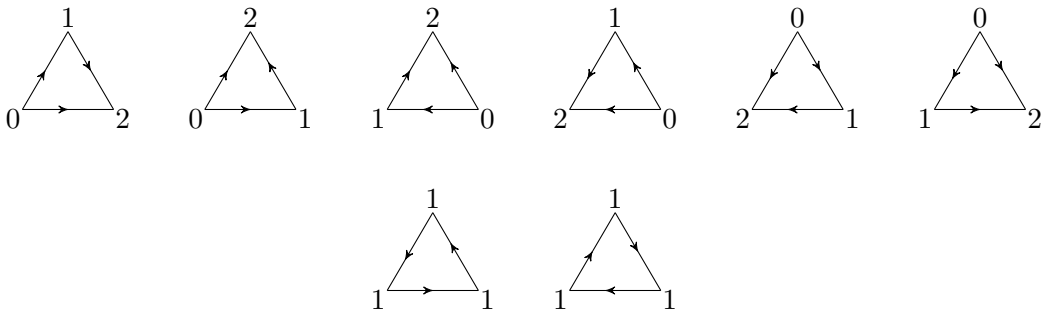
\begin{example}\label{triangleEx1}
Figure \ref{triangle} depicts $8$ possible orientations of a triangle and $7$ corresponding indegree divisors.
\end{example}
\begin{remark}
Divisors satisfying Definition \ref{balDef1} are also known in literature as score vectors~\cite{KW} and ordered outdegree sequences \cite{stanley1980decompositions} (note that there is no difference between indegree and outdegree divisors, since $\indeg({\mathfrak o}) = \mathrm{outdeg}({-{\mathfrak o}})$ where $-{\mathfrak o}$ is the orientation obtained from $\mathfrak o$ by inverting the direction of every edge). In algebraic geometry, divisors satisfying Definition~\ref{balDef1} are known as  \textit{normalized semistable multidegrees}. Such multidegrees were used by Beauville~\cite{Beauville} to describe the theta divisor for a nodal reducible curve. The term normalized semistable multidegree is due to Alexeev~\cite{Alexeev}. Let us also mention a recent paper~\cite{ABKS} where the authors consider so-called \textit{orientable} divisors. A divisor $D \in \Div(\Gamma)$ is called orientable if there exists an orientation $\mathfrak o \in \pazocal O(\Gamma)$  such that
$
D = \sum_{v} (\indeg_{\mathfrak o}(v) - 1)v.
$
Clearly, orientable divisors are in bijection with indegree divisors.
\end{remark}
The following result is well-known.
  \begin{proposition}\label{balDef2}
A divisor $D$ on a graph $\Gamma$ is {an indegree divisor} if and only if the degree of $D$ is equal to the number of edges of $\Gamma$, and for every subgraph $\Gamma_1 \subset \Gamma$, the degree of the restriction of $D$ to $\Gamma_1$ is greater or equal than the number of edges of $\Gamma_1$.
  \end{proposition}
  For the proof of Proposition \ref{balDef2}, see e.g.~\cite{hakimi, gyarfas1976orient, Alexeev}. See also~\cite{backman} where it is shown that this result is equivalent to the well-known max-flow min-cut theorem in optimization theory.\par
  We will also need one more description of indegree divisors. To each vertex $v_i$ of $\Gamma$, we associate a formal variable $x_i = \exp({v_i})$, and consider the polynomial
  $$
  B_\Gamma(x_1, \dots, x_n) := \prod_{[v_i, v_j]} (x_i + x_j)
  $$
  where the product is taken over all edges of $\Gamma$. We have
 \begin{align}\label{BPolynomial}
    B_\Gamma =\prod_{[v_i, v_j]} (\exp({v_i})+ \exp({v_j}))=\!\!\!  \!\!\! \sum_{D \,\in\, \Div(\Gamma)} \!\!\!\!\!  \mult(\Gamma, D) \exp(D),
  \end{align}
   where $\mult(\Gamma, D)$ is a non-negative integer for every divisor $D$. 
   The following proposition is straightforward:
   \begin{proposition}\label{balDef3}
 A divisor $D = \sum d_iv_i   \in \Div(\Gamma)$ is an indegree divisor if and only if the polynomial $B_\Gamma$ contains the monomial $\exp(D) = x_1^{d_1}\cdots x_k^{d_k}, $ or, equivalently, if and only if $\mult(\Gamma, D) > 0$. Moreover, if $D$ is an indegree divisor, then the number $\mult(\Gamma, D)$ is equal to the number of orientations $\mathfrak o \in \pazocal O(\Gamma)$ such that $D = \indeg(\mathfrak o)$.     \end{proposition}
% \begin{definition}
%   We call the number $m(d)$ the \textit{multiplicity} of a orientable divisor $d$.
%   \end{definition}
%   Note that
%   $$
%    \sum_{d \,\in\, \Div(\Gamma)} \!\!\!m(d) =\mbox{total number of orientations of $\Gamma$} =2^{|E(\Gamma)|}
%   $$
%   where $E(\Gamma)$ is the set of edges of $\Gamma$.
%  \begin{example}
%  Assume that $\Gamma$ has two vertices joined by $n$ edges. Then
%  $$
%  B_\Gamma = (x_1 + x_2)^n = \sum_{i = 0}^n \binom{n}{i} x_1^i x_2^{n-i},
%  $$
%  so every positive divisor on $\Gamma$ is balanced.
%  \end{example}
   \begin{example}\label{triangleEx2}
   Let $\Gamma$ be a triangle. Then
\begin{align}\label{triangleB}
\begin{aligned}
   B_\Gamma&= (x_1 + x_2)(x_1 + x_3)(x_2 + x_3) = \\  &x_1^2 x_2 + x_1x_2^2 + x_1^2x_3 + x_1x_3^2 + x_2^2x_3 + x_2x_3^2 + 2x_1x_2x_3, 
   \end{aligned}
   \end{align}
which again shows that there are $7$ indegree divisors on $\Gamma$ (cf. Figure \ref{triangle}).
   \end{example}
   Now, consider the space $\mathrm{span}_\R \langle v_1, \dots, v_k \rangle = \Div(\Gamma) \otimes \R$
 formally spanned over $\R$ by vertices of $\Gamma$. Note that each edge of $\Gamma$ may be viewed as a line segment in this space. The following definition is due to T.\,Zaslavsky, see \cite{postnikov}:
 \begin{definition}
 Let $\Gamma$ be a graph. Then the \textit{graphical zonotope} $Z_\Gamma$ is the Minkowski sum of edges $[v_i, v_j] \in \Gamma$ considered as line segments in the space $\mathrm{span}_\R \langle v_1, \dots, v_k \rangle$. 
 \end{definition}
 \begin{statement}\label{balDef4}
 For any graph $\Gamma$, indegree divisors on $\Gamma$ are exactly lattice points in the corresponding graphical zonotope $Z_\Gamma$.
 \end{statement}
 \begin{proof}
  It is well known that the Newton polytope of a product is the Minkowski sum of Newton polytopes of factors. Therefore, the graphical zonotope $Z_\Gamma$ is the Newton polytope of the polynomial $B_\Gamma$. In other words, $Z_\Gamma$ is the convex hull of all indegree divisors. In particular, any indegree divisor is a lattice point in  $Z_\Gamma$. Conversely, let $D$ be a lattice point in $Z_\Gamma$. Then $D$ lies in the convex hull of indegree divisors, and, as follows from  Proposition \ref{balDef2}, it is itself an indegree divisor, q.e.d.
% 
% Proposition \ref{balDef2} implies that the set of indegree divisors on a given graph $\Gamma$ is the set of lattice points in a certain convex polytope $\Delta_\Gamma$.  It can be shown that this polytope has integral vertices and, therefore, it coincides with the convex hull of all indegree divisors. In other words, the polytope $\Delta_\Gamma$ is the Newton polytope of the polynomial $B_\Gamma$. On the other hand, it is well known that the Newton polytope of a product is the Minkowski sum of Newton polytopes of factors. so the Newton polytope of $B_\Gamma$ is the polytope $Z_\Gamma$.  Therefore, $B_\Gamma = Z_\Gamma$, which implies the proposition.
\end{proof}
Now we have four different descriptions of indegree divisors: in terms of orientations, in terms of linear inequalities, in terms of monomials in the polynomial $B_\Gamma$, and in terms of lattice points in the graphical zonotope $Z_\Gamma$.
 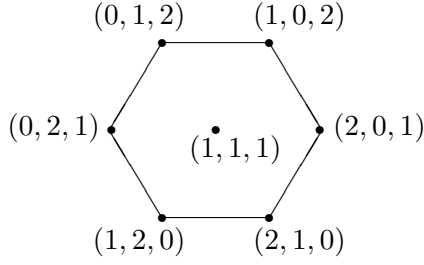
\begin{figure}[t]
\centerline{
\begin{picture}(100,100)
\put(50,50){
\begin{picture}(200,200)
\put(-20, 33){\line(1,0){40}}
\put(20, 33){\line(3,-5){20}}
\put(20, -33){\line(3,5){20}}
\put(-20, -33){\line(1,0){40}}
\put(-20, 33){\line(-3,-5){20}}
\put(-20, -33){\line(-3,5){20}}
\put(39,0){\circle*{3}}
\put(44,-2){$(2,0,1)$}
\put(14,-45){$(2,1,0)$}
\put(-46,-45){$(1,2,0)$}
\put(-78,-2){$(0,2,1)$}
\put(-46,40){$(0,1,2)$}
\put(14,40){$(1,0,2)$}
\put(-10,-10){$(1,1,1)$}
\put(-39,0){\circle*{3}}
\put(20,33){\circle*{3}}
\put(-20,33){\circle*{3}}
\put(20,-33){\circle*{3}}
\put(-20,-33){\circle*{3}}
\put(0,0){\circle*{3}}
\end{picture}
}
\end{picture}
}
\caption{Lattice points in the permutohedron $P_3$}\label{permut3}
\end{figure}
 \begin{example}\label{completeGraph}
 Let $\Gamma = K_n$ be the complete graph on $n$ vertices. Then 
 $$
  Z_\Gamma =\sum_{i < j} \,\Newton(x_i + x_j) = \sum_{i < j}\, \Newton(x_i - x_j) = \Newton(\mathrm {V}(x_1, \dots, x_n))
 $$
 where $\Newton(B)$ stands for the Newton polytope of a polynomial $B$, and $\mathrm {V}$ is the Vandermonde determinant
 $$
\mathrm {V}(x_1, \dots, x_n) =  \prod_{i < j}\, (x_i - x_j) = \det(x_i^{j-1}) = \sum_{\sigma \in S_n} (-1)^\sigma x_1^{\sigma(0)} \dots x_n^{\sigma(n-1)},
 $$
 and $S_n$ denotes all permutations of $(0,1, \dots, n-1)$. So, the graphical zonotope for the complete graph $K_n$ is the convex hull of all permutations $\sigma \in S_n$. This polytope is known as the \textit{permutohedron} $P_n$. Figure \ref{permut3} depicts the permutohedron $P_3$ corresponding to the triangle $K_3$. Lattice points in this permutohedron are seven indegree divisors (cf. Example \ref{triangleEx1} and Example~\ref{triangleEx2}).
  \end{example}
% \begin{remark}
% Note that Definition \ref{balDef2} allows one to describe the graphical zonotope $Z_\Gamma$ in terms of linear inequalities. For the permutohedron $P_n$, this result is due to R.\,Rado \cite{Rado}.
% \end{remark}
 \begin{remark}\label{verticesOfZG}
 Another way to define the graphical zonotope $Z_\Gamma$ is to consider the convex hull of indegree divisors of all acyclic orientations. Moreover, indegree divisors of the form $\indeg(\mathfrak o)$ where $\mathfrak o$ is acyclic are precisely vertices of $Z_\Gamma$ (cf. \cite{Stanley}, Exercise 4.32). Note that if the graph $\Gamma$ has loops, then it does not admit acyclic orientations. In this case, vertices of the graphical zonotope $Z_\Gamma$  are indegree divisors of those orientations that have no oriented cycles except compositions of loops.
 \end{remark}
 \section{Description of strata}\label{secDOS}
 In Section \ref{genCons}, we constructed a stratification
 $
 \pazocal P^B_C = \bigsqcup\,\Str(\Gamma ,D)
 $
 where $\Gamma$ is a generating subgraph of the dual graph $\Gamma_C$ of $C$, and $D$ is a divisor on $\Gamma$.  Now, we describe those divisors $D$ for which the stratum $\Str(\Gamma ,D)$ is non-empty.\par
% Define the following reference divisor $D_0$ on $\Gamma_C$:
%$$
%D_0 :=\!\!\! \sum_{v_i \in V(\Gamma_C)}\! (g_i + n_i - 1)v_i
%$$
%where $g_i$ is the geometric genus of the irreducible component $C_i$ of $C$ corresponding to the vertex $v_i$, and $n_i$ is the number of infinite points on $C_i$, which is equal to the degree of the equation of $C_i$ in the variable $\mu$.\par
For a generating subgraph $\Gamma$ of the dual graph $\Gamma_C$, denote by ${C}_{\Gamma}$ the curve obtained from $C$ by resolving those nodes which correspond to edges $e \in \Gamma_C \setminus \Gamma$. Let also ${C}_{\Gamma} \, / \, \infty$ be the curve obtained from ${C}_{\Gamma}$ by identifying $n$ points at infinity. %For a divisor $d \in \Div(\Delta)$, denote by $\Pic_d({C}_{\Delta} \, / \, \infty)$ the multidegree $d$ Picard variety of ${C}_{\Delta} \, / \, \infty$.
\begin{theorem}\label{thm1}
Assume that $C$ is a nodal curve satisfying condition \eqref{specCond}. Let also $\Gamma$ be a generating subgraph of the dual graph $\Gamma_C$ of $C$, and let $D \in \Div(\Gamma)$. Then 
\begin{longenum} \item the stratum $\Str(\Gamma ,D)$ is non-empty if and only if $D$ is an indegree divisor on $\Gamma$;
\item each non-empty stratum $\Str(\Gamma ,D)$ is a smooth irreducible quasi-affine variety biholomorphic to an open dense subset in the generalized Jacobian of the curve ${C}_{\Gamma} \, / \, \infty$. In particular,
\begin{align*}%\label{dimensionFormula}
\dim \Str(\Gamma ,D)=  \frac{mn(n-1)}{2} - |E(\Gamma_C)| + |E(\Gamma)|;
\end{align*}
\item for each non-empty stratum $\Str(\Gamma ,D)$, the quotient $\Str(\Gamma ,D) / G^B$ is biholomorphic to an open dense subset in the generalized Jacobian of the curve ${C}_{\Gamma}$.
%\end{longenum}
%\item if $D$ is a balanced divisor on $\Gamma$, then the quotient of $\Str(\Gamma ,D)$ with respect to the conjugation action of the centralizer $G^B$ of $B$ is biholomorphic to an open subset in the generalized Jacobian of the curve ${C}_{\Gamma}$.
\end{longenum}
\end{theorem}
\begin{corollary}
We have $$\dim \pazocal P^B_C = \dfrac{mn(n-1)}{2}. $$
\end{corollary}
\begin{remark}
Details on generalized Jacobians of singular curves can be found in \cite{rosenlicht, serre}. Here we recall that the generalized Jacobian $\Jac(C)$ of a nodal curve $C$ is an extension of the Jacobian of the normalization of $C$ with a commutative algebraic group $(\Complex^*)^k$, where $k$ is a non-negative integer. The same is true for a nodal curve with identified points at infinity, provided that these points are all distinct. In particular, if $C$ is a rational nodal curve, then $\Jac(C) \simeq (\Complex^*)^k$, and $\Jac(C /\infty ) \simeq (\Complex^*)^m$, where $m \geq k$. %Also note that the dimension of the generalized Jacobian of a connected curve is equal to its arithmetic genus. For the curve $C_\Gamma / \infty$, the arithmetic genus can be computed by formula \eqref{dimensionFormula}.
\end{remark}
The proof of Theorem \ref{thm1} is based on careful analysis of the correspondence between matrix polynomials and line bundles described in Section \ref{genCons} and is similar to Gavrilov's proof in the smooth case \cite{Gavrilov2}. Details of the proof will be published elsewhere.
%\begin{example}
%Assume that $C$ is irreducible with $k$ nodes. Then $\pazocal P^B_C$ has exactly one stratum in each dimension 
%$$
% \frac{1}{2}mn(n-1) - k \leq dim \leq  \frac{1}{2}mn(n-1).
%$$
%\end{example}
\begin{example}\label{twoLines5}
Let $C$ be two lines, as in Example \ref{twoLines} and Examples \ref{twoLines2} -  \ref{twoLines4_5}. Then
\begin{align*} \Str(
   \begin{tikzpicture}[baseline={([yshift=-.5ex]current bounding box.center)},vertex/.style={anchor=base,
    circle,fill=black!25,minimum size=18pt,inner sep=2pt}]
            \fill (0,0) circle (1.2pt);
            \fill (1,0) circle (1.2pt);
   		 \draw  [] (0,0) -- (1,0);    
		\node () at (-0.2,-0) {$0$};
		\node () at (1.2,-0) {$1$};	
            \end{tikzpicture}) \simeq \Str(
   \begin{tikzpicture}[baseline={([yshift=-.5ex]current bounding box.center)},vertex/.style={anchor=base,
    circle,fill=black!25,minimum size=18pt,inner sep=2pt}]
            \fill (0,0) circle (1.2pt);
            \fill (1,0) circle (1.2pt);
   		 \draw  [] (0,0) -- (1,0);    
		\node () at (-0.2,-0.) {$1$};
		\node () at (1.2,-0.) {$0$};	
            \end{tikzpicture}) \simeq \Complex^*,\quad
            \Str(
   \begin{tikzpicture}[baseline={([yshift=-.5ex]current bounding box.center)},vertex/.style={anchor=base,
    circle,fill=black!25,minimum size=18pt,inner sep=2pt}]
            \fill (0,0) circle (1.2pt);
            \fill (1,0) circle (1.2pt);
   		 \draw  [dashed] (0,0) -- (1,0);    
		\node () at (-0.2,-0.) {$0$};
		\node () at (1.2,-0.) {$0$};	
            \end{tikzpicture}) \simeq \mbox{a point},
\end{align*}
which means that for all strata we have an isomorphism $\Str(\Gamma ,D) \simeq \Jac(C_\Gamma / \infty)$, i.e. in this example the open dense subset from Theorem \ref{thm1} is the whole Jacobian.
%%Let us again consider Example \ref{twoLines}, i.e. let $m =1$, $n=2$, $C= \{\mu = a_1 + b_1 \lambda\} \cup \{\mu = a_2 + b_2 \lambda\}$, and $J= \diag(b_1, b_2)$.  Let us show that the stratification of $ \pazocal P^B_C$ given by Theorem~\ref{thm1} coincides with the stratification $\mathrm{Orb}_1 \sqcup \mathrm{Orb}_2 \sqcup \mathrm{Orb}_3$. 
%%Indeed, the dual graph $\Gamma_C$ is two vertices joined by an edge. Its generating subgraphs are the graph $\Gamma_C$ itself, and the graph $\emptyset$ which has two disjoint vertices. There are two balanced divisors on $\Gamma_C$, namely $(1,0)$ and $(0,1)$ (see Figure \ref{diangle}, (b) and (c)), so $\pazocal P^B_C$ has two strata of dimension one: $\Str(\Gamma_C ,(1,0))$ and $\Str(\Gamma_C ,(0,1))$ (note that the reference divisor $D_0$ is equal to zero). As for the subgraph $\emptyset$, its only balanced divisor is the zero divisor, so  the only zero-dimensional stratum of $ \pazocal P^B_C$ is $\Str(\emptyset, 0)$ that is a point $\diag(a_1 + b_1\lambda, a_2 + b_2 \lambda)$.
%%
%
%%: they are the sets $\mathrm{Orb}_2$ and $\mathrm{Orb}_3$ from Example~\ref{twoLines} (cf. Example \ref{twoLines3}). 
%%Note that $\mathrm{Orb}_2 \simeq \mathrm{Orb}_3 \simeq \Complex^*$, so these strata are isomorphic to the whole generalized Jacobian of ${C} \, / \, \infty$, not just a dense subset.
%%Note that is this example for any generating subgraph $\Gamma \subset \Gamma_C$ and any $D$balanced $D \in \Div(\Gamma)$, we have $$
\end{example}
\begin{figure}[t]
\centerline{
\begin{tikzpicture}
%\node (O) at (-5,-2.4)
%{
%\begin{tikzpicture}
%\draw (0,0) -- (1.5,2);
%\draw (2,0) -- (0.5,2);
%\draw (-0.2,0.5) -- (2.2,0.5);
%\end{tikzpicture}
%};
\node () at (-7,-2.4)
{
\begin{tikzpicture}[scale = 0.5]
\node (A) at (0,0) {
$ (0,1,2)$
};
\node (B) at (3.5,0)
{
$
(1,0,2)
$
};
\node (C) at (5, -2.4)
{
$
(2,0,1)
$
};
%\node (G) at (1.75, -2.4)
%{
%$\Str\left(
%            \begin{tikzpicture}[scale = 0.8, baseline={([yshift=-.5ex]current bounding box.center)},vertex/.style={anchor=base,
%    circle,fill=black!25,minimum size=18pt,inner sep=2pt}]
%   		 \draw   (0,0) -- (0.5, 0.86);
%   		 \draw   (0,0) -- (1,0);    
% 		   \draw    (0.5, 0.86) -- (1,0); 
%		\node () at (-0.15,-0.1) {$1$};
%		\node () at (1.15,-0.1) {$1$};	
%				\node () at (0.5,1.1) {$1$};
%            \end{tikzpicture} 
%\right) $
%};
\node (D) at (3.5,-4.8)
{
$
(2,1,0)
$
};
\node (E) at (0,-4.8)
{
$
(1,2,0)
$
};
\node (F) at (-1.5,-2.4)
{
$
(0,2,1)
$
};
\draw (A) -- (B) -- (C) -- (D) --(E) -- (F) -- (A);
\end{tikzpicture}
};
\draw [dashed, ->] (-4.3,-2.4) -- (-3.3,-2.4);
\node (A) at (0,0) {
$ \left(\begin{array}{ccc}\nu_1 & * & * \\0 & \nu_2 & * \\0 & 0 & \nu_3\end{array}\right)$
};
\node (B) at (3.5,0)
{
$
 \left(\begin{array}{ccc}
\nu_1 & 0 & * \\ 
* & \nu_2 & * \\
 0 & 0 & \nu_3
 \end{array}\right)
$
};
\node (C) at (5, -2.4)
{
$
 \left(\begin{array}{ccc}
\nu_1 & 0 & 0 \\ 
* & \nu_2 & * \\
 * & 0 & \nu_3
 \end{array}\right)
$
};
%\node (G) at (1.75, -2.4)
%{
%$\Str\left(
%            \begin{tikzpicture}[scale = 0.8, baseline={([yshift=-.5ex]current bounding box.center)},vertex/.style={anchor=base,
%    circle,fill=black!25,minimum size=18pt,inner sep=2pt}]
%   		 \draw   (0,0) -- (0.5, 0.86);
%   		 \draw   (0,0) -- (1,0);    
% 		   \draw    (0.5, 0.86) -- (1,0); 
%		\node () at (-0.15,-0.1) {$1$};
%		\node () at (1.15,-0.1) {$1$};	
%				\node () at (0.5,1.1) {$1$};
%            \end{tikzpicture} 
%\right) $
%};
\node (D) at (3.5,-4.8)
{
$
 \left(\begin{array}{ccc}
 \nu_1 & 0 & 0 \\ 
* & \nu_2 & 0 \\ 
 * & * & \nu_3
 \end{array}\right)
$
};
\node (E) at (0,-4.8)
{
$
  \left(\begin{array}{ccc}
 \nu_1 & * & 0 \\ 
 0 & \nu_2 & 0 \\ 
 * & * & \nu_3
 \end{array}\right)
$
};
\node (F) at (-1.5,-2.4)
{
$
   \left(\begin{array}{ccc}
 \nu_1 & * & * \\ 
0& \nu_2 & 0 \\ 
 0 & * & \nu_3
 \end{array}\right)
$
};
\draw (A) -- (B) -- (C) -- (D) --(E) -- (F) -- (A);
%\draw (A) -- (G) -- (B);
%\draw (C) -- (G) -- (D);
%\draw (E) -- (G) -- (F); 
\end{tikzpicture}
}
\caption{Strata corresponding to vertices of the permutohedron $P_3$.}\label{vertComp}
\end{figure}
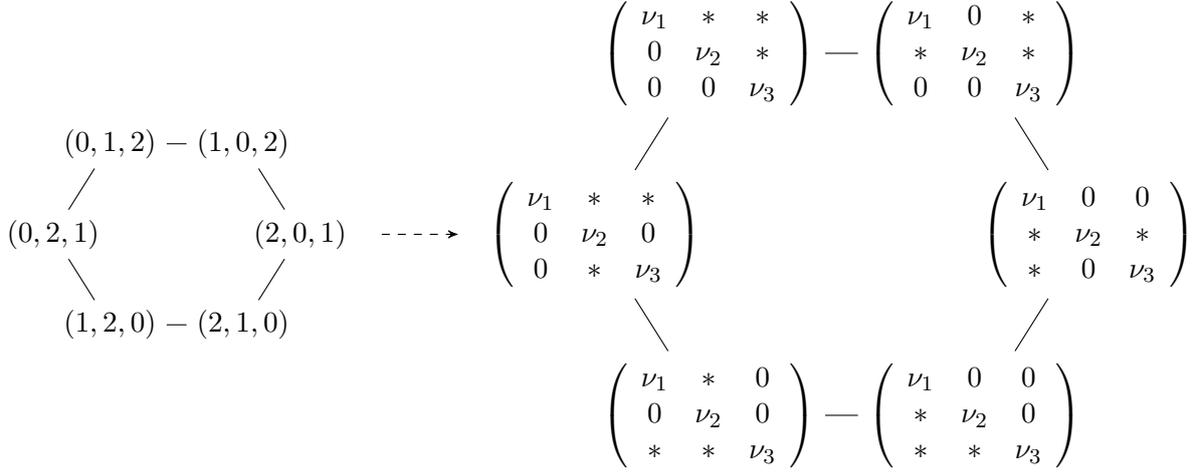
\begin{example}\label{threeLines}
Let $m =1$, $n=3$, and let $C$ be the union of three straight lines $C_i = \{\mu = a_i + b_i \lambda \}$ in general position. Let also $B= \diag(b_1, b_2, b_3)$. The dual graph $\Gamma_C$ is a triangle, so $|\BDiv(\Gamma_C)| = 7$  (see Figures \ref{triangle}, \ref{permut3}). Accordingly, the variety $ \pazocal P^B_C $ has seven strata of dimension three. Six of them correspond to vertices of the permutohedron $P_3$; these six strata are depicted in Figure \ref{vertComp}. 
% Each vertex of the hexagon in Figure \ref{vertComp} represents the stratum associated with the corresponding vertex of the hexagon in Figure~\ref{permut3}.  
The notation $\nu_i$ stands for $ a_i + b_i \lambda $, and the numbers denoted by stars are assumed to be non-zero. Note that there is a one-to-one correspondence between these strata and Borel subalgebras of $\gl_3$ that contain the Cartan subalgebra of diagonal matrices.
%The question mark denotes the seventh stratum corresponding to the interior lattice point of the hexagon in Figure~\ref{permut3}.
\par
Now, let us describe the seventh three-dimensional stratum $\Str(\Gamma_C, v_1 + v_2 +v_3)$ corresponding to the interior lattice point $(1,1,1)$ in the permutohedron $P_3$. Let $$(c_1,c_2,c_3) := (1,1,1) \times (b_1, b_2,b_3),\quad k := (c_1,c_2,c_3) \cdot (a_1, a_2, a_3)$$
where $\times$ is the cross product, and dot denotes the inner product. Then the stratum $\Str(\Gamma_C, v_1 + v_2 +v_3)$ consists of matrix polynomials of the form $A + \lambda B$ where $a_{ii} = a_i$, off-diagonal entries of the matrix $A$ satisfy
%\begin{align}\label{intStratum}
%P(\lambda) = 
%\left(\begin{array}{ccc}
% \nu_1 & x_{12} & x_{13} \\ 
% x_{21} & \nu_2 & x_{23} \\ 
% x_{31} & x_{32} & \nu_3
% \end{array}\right)
%\end{align}
%where $x_{ij} \in \Complex^*$ satisfy
\begin{align*}
%\label{intStratum}
a_{12}a_{21} = c_3 z,\quad  a_{13}a_{31} = c_2 z, \quad a_{23}a_{32} = c_1 z, \quad a_{12}a_{23}a_{31} = w,\end{align*}
and $(z,w) \neq (0,0)$ is any point lying in the affine part of the nodal cubic
\begin{align}\label{ratCubic}
w(kz - w) = c_1c_2c_3z^3.
\end{align}

\par
 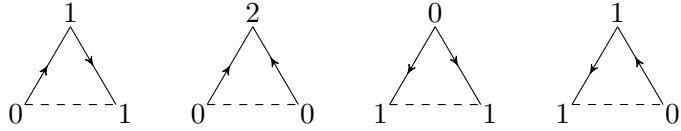
\begin{figure}[t]
\centerline{
\begin{tikzpicture}[thick,scale = 1.2]
 \node (a) at (0,0)
         {
            \begin{tikzpicture}[scale = 1.2]
   		 \draw  [->-] (0,0) -- (0.5, 0.86);
   		 \draw  [dashed] (0,0) -- (1,0);    
 		   \draw  [->-]  (0.5, 0.86) -- (1,0); 
		\node () at (-0.1,-0.1) {$0$};
		\node () at (1.1,-0.1) {$1$};	
				\node () at (0.5,1.02) {$1$};	
            \end{tikzpicture}
         }; 
          \node (b) at (2,0)
         {
            \begin{tikzpicture}[scale = 1.2]
   		 \draw  [->-] (0,0) -- (0.5, 0.86);
   		 \draw  [dashed] (0,0) -- (1,0);    
 		   \draw  [-<-]  (0.5, 0.86) -- (1,0); 
		\node () at (-0.1,-0.1) {$0$};
		\node () at (1.1,-0.1) {$0$};	
				\node () at (0.5,1.02) {$2$};	
            \end{tikzpicture}
         }; 
                   \node (c) at (4,0)
         {
            \begin{tikzpicture}[scale = 1.2]
   		 \draw  [-<-] (0,0) -- (0.5, 0.86);
   		 \draw  [dashed] (0,0) -- (1,0);    
 		   \draw  [->-]  (0.5, 0.86) -- (1,0); 
		\node () at (-0.1,-0.1) {$1$};
		\node () at (1.1,-0.1) {$1$};	
				\node () at (0.5,1.02) {$0$};	
            \end{tikzpicture}
         }; 
                            \node (d) at (6,0)
         {
            \begin{tikzpicture}[scale = 1.2]
   		 \draw  [-<-] (0,0) -- (0.5, 0.86);
   		 \draw  [dashed] (0,0) -- (1,0);    
 		   \draw  [-<-]  (0.5, 0.86) -- (1,0); 
		\node () at (-0.1,-0.1) {$1$};
		\node () at (1.1,-0.1) {$0$};	
				\node () at (0.5,1.02) {$1$};				
            \end{tikzpicture}
         }; 
\end{tikzpicture}
}
\caption{Four indegree divisors on two sides of a triangle}\label{path}
\end{figure}
Further, let us describe two-dimensional strata. There are three generating subgraphs of $\Gamma_C$ with two edges, and each of them has four indegree divisors (see Figure \ref{path}), so $ \pazocal P^B_C$ has $3 \times 4 = 12$ two-dimensional strata. These strata look similarly to strata in Figure~\ref{vertComp}, but with two stars (i.e., two non-zero elements) instead of three. Stars can be placed at any two off-diagonal positions which are not symmetric with respect to the main diagonal; there are exactly $12$ ways to choose such two positions. Similarly, there are six one-dimensional strata that can be obtained by placing one star at any position, and one zero-dimensional stratum, the point $\diag(\nu_1, \nu_2, \nu_3)$.
\end{example}
\begin{remark}
Note that the stratum $\Str(\Gamma_C, v_1 + v_2 +v_3)$ in the previous example is isomorphic to $(\Complex^*)^2 $ times a thrice punctured sphere, which means that the open dense subset from Theorem~\ref{thm1} in this case is \textit{not} the whole Jacobian.  For any other stratum of $ \pazocal P^B_C$, we have an isomorphism $\Str(\Gamma ,D) \simeq \Jac(C_\Gamma / \infty)$, as in Example \ref{twoLines5}. 
%Note that $c_1c_2c_3 \neq 0$ and $d \neq 0$ since the lines $l_1$, $l_2$, $l_3$ are in general position.\par
%Also note that for all strata of $ \pazocal P^B_C$ except the stratum \eqref{int\Stratum}, we have an isomorphism $\Str(\Gamma ,D) \simeq \Jac(C_\Gamma / \infty)$, as in Example \ref{twoLines5}. However, the stratum $\Str(\Gamma_C, v_1 + v_2 + v_3)$ given by \eqref{int\Stratum} is $(\Complex^*)^2$ times a thrice-punctured sphere, therefore, it is isomorphic to a proper dense subset of $\Jac(C / \infty)$.
\end{remark}
\begin{remark}
Note that matrix polynomials lying in opposite vertices of the hexagon in Figure~\ref{vertComp} are transposes of each other. More generally, if $B$ is a symmetric matrix, then there is an involution $\sigma$ on $\pazocal P^B_C$ given by $\sigma(P) = P^t$. This involution maps each stratum $\Str(\Gamma,D)$ to another stratum $\Str(\Gamma, D')$ and thus induces an involution $\tau \colon D \mapsto D'$ on the set  $\BDiv(\Gamma)$. Explicitly, the involution $\tau$ can be defined by the formula $\tau(\indeg(\mathfrak o)) = \indeg(-{\mathfrak  o})$ where $-{\mathfrak  o}$ is the orientation inverse to $\mathfrak o$. In other words, we have $\tau(D) = \deg(\Gamma) - D$, where $\deg(\Gamma) := \sum_{v \in V(\Gamma)} \deg(v)v$ is the degree divisor of the graph $\Gamma$. Equivalently, the involution $\tau$ can be described as the central symmetry of the graphical zonotope $Z_\Gamma$.\end{remark}

\section{Local structure and irreducible components}

In this section, we describe the local structure of the variety $\pazocal P^B_C$ in the neighborhood of its smooth stratum $\Str(\Gamma ,D)$. Let a \textit{node} be the germ at the origin of the complex-analytic variety $\{(z,w) \in \Complex^2 \mid zw = 0\}$. %Denote a node by $N$.

\begin{theorem}\label{thm2}
Let $P \in \Str(\Gamma ,D)$. Then, in the neighborhood of the point $P$, the isospectral variety $\pazocal P^B_C$ is locally isomorphic to the direct product of $\codim \Str(\Gamma ,D) =  |E(\Gamma_C)| - |E(\Gamma)|$ nodes\footnote{We use the notation $\codim \Str(\Gamma ,D)$ for the codimension of the stratum $ \Str(\Gamma ,D)$ in $\pazocal P^B_C$, i.e. $\codim \Str(\Gamma ,D) := \dim \pazocal P^B_C - \dim \Str(\Gamma ,D)$.} and a smooth disk of dimension  $\dim \Str(\Gamma ,D) = \dim \pazocal P^B_C - \left\lvert E(\Gamma_C)\right\rvert + |E(\Gamma)|.$
%$$
%N^{|E(\Gamma_C)| - |E(\Gamma)|} \times \Complex^{ \frac{1}{2}mn(n-1) - |E(\Gamma_C)| + |E(\Gamma)|}
%$$
\end{theorem}
\begin{example}
Let $C$ be two lines, as in Examples \ref{twoLines} and \ref{twoLines2} - \ref{twoLines4_5}, and let $\emptyset \subset \Gamma_C$ be a generating subgraph with no edges, i.e. the disjoint union of vertices of $\Gamma_C$. The only indegree divisor on the graph $\emptyset$ is the zero divisor, and the corresponding stratum $\Str(\emptyset ,0)$ is one point $\diag(a_1+b_1\lambda, a_2+b_2\lambda)$. Theorem \ref{thm2} implies that near this point $\pazocal P^J_{C}$ is locally a node. Two disks forming this node are closures of the strata $ \Str(
   \begin{tikzpicture}[baseline={([yshift=-.5ex]current bounding box.center)},vertex/.style={anchor=base,
    circle,fill=black!25,minimum size=18pt,inner sep=2pt}]
            \fill (0,0) circle (1.2pt);
            \fill (1,0) circle (1.2pt);
   		 \draw  [] (0,0) -- (1,0);    
		\node () at (-0.2,-0) {$0$};
		\node () at (1.2,-0) {$1$};	
            \end{tikzpicture})$ and $\Str(
   \begin{tikzpicture}[baseline={([yshift=-.5ex]current bounding box.center)},vertex/.style={anchor=base,
    circle,fill=black!25,minimum size=18pt,inner sep=2pt}]
            \fill (0,0) circle (1.2pt);
            \fill (1,0) circle (1.2pt);
   		 \draw  [] (0,0) -- (1,0);    
		\node () at (-0.2,-0.) {$1$};
		\node () at (1.2,-0.) {$0$};	
            \end{tikzpicture}) $.
\end{example}
\begin{example}\label{threeLinesMult0}
Let $C$ be three lines, as in Example \ref{threeLines}. Then the stratum $\Str(\emptyset ,0)$ is one point $\diag(\nu_1, \nu_2, \nu_3)$ where $\nu_i = a_i + b_i \lambda$. Theorem \ref{thm2} implies that near this point the isospectral variety $\pazocal P^J_{C}$ is locally a product of three nodes, i.e. a union of eight three-dimensional disks. Six of these disks are closures of strata depicted in Figure~\ref{vertComp}, and other two belong to the closure of the seventh stratum $\Str(\Gamma_C, v_1 + v_2 + v_3)$ (note that this closure has a double point at the origin corresponding to a double point of the curve~\eqref{ratCubic}).
\end{example}
\begin{corollary}[of Theorem \ref{thm2}]
Irreducible components of the variety $ \pazocal P^B_C$ are in one-to-one correspondence with indegree divisors on the dual graph $\Gamma_C$.
\end{corollary}
\begin{proof}
According to Theorem \ref{thm2}, each stratum $\Str(\Gamma, D)$ where $\Gamma$ is a proper subgraph of $\Gamma_C$ lies in the closure of a higher-dimensional stratum. Therefore, irreducible components of $ \pazocal P^B_C$ are closures of strata of the form $\Str(\Gamma_C,D)$, as desired.
\end{proof}
\begin{example}\label{nLines}
Let $m =1$, and let $C$ be the union of $n$ straight lines $\{\mu = a_i + b_i \lambda \}$ in general position. Let also $B= \diag(b_1, \dots, b_n)$. The dual graph $\Gamma_C$ is the complete graph $K_n$ on $n$ vertices. Therefore, irreducible components of $ \pazocal P^B_C$ are enumerated by lattice points in the permutohedron $P_n$ (see Example \ref{completeGraph}). It is easy to give an explicit description of components corresponding to vertices of the permutohedron: they generalize components depicted in Figure \ref{vertComp} and correspond to $n!$ Borel subalgebras of $\gl_n$ containing diagonal matrices. Explicit description of other components for $n > 3$ is unknown. \par
Also note that the number of lattice points in the permutohedron $P_n$ is known to be equal to the number of forests on $n$ labeled vertices. It is sequence A001858 in the online encyclopedia of integer sequences: $1,2,7,38, 291, \dots$.
\end{example}
\begin{remark}
The idea of the proof of Theorem \ref{thm2} is to show that if $C$ is a nodal curve, then every $P \in \pazocal P^B_C$ is a non-degenerate singular point for the integrable system on matrix polynomials (for the definition of non-degenerate singular points, see, e.g.,~\cite{bolosh}, Definition 7). Then one applies J.\,Vey's normal form theorem \cite{Vey}.
\end{remark}

  \section{Combinatorial interlude II: Multiplicities}
\begin{definition}\label{multDef}
Let $\Gamma$ be a graph, and let $D$ be an indegree divisor on $\Gamma$. The \textit{multiplicity} of $D$ is the number $\mult(\Gamma, D)$ entering formula \eqref{BPolynomial} for the polynomial $B_\Gamma$ or, equivalently, the number of orientations $\mathfrak o \in \pazocal O(\Gamma)$ such that $D = \indeg(\mathfrak o)$\footnote{Equivalence of these definitions follows from Proposition \ref{balDef3}.}.
\end{definition}
%\begin{example}
%\begin{align*}
%  \mult\left( \begin{tikzpicture}[scale = 1.7, baseline={([yshift=-.5ex]current bounding box.center)},vertex/.style={anchor=base,
%    circle,fill=black!25,minimum size=18pt,inner sep=2pt}]
%   		 \draw  [] (0,0) -- (0.5, 0.86);
%		 \draw [] (0,0) -- (0.5, 0.3);
%		 \draw (0.5, 0.3) -- (1,0);
%		  \draw (0.5, 0.3) -- (0.5, 0.86);
%   		 \draw  [] (0,0) -- (1,0);    
% 		   \draw  []  (0.5, 0.86) -- (1,0); 
%		\node () at (-0.1,-0.1) {$2$};
%		\node () at (1.1,-0.1) {$2$};	
%				\node () at (0.5,1.02) {$1$};	
%				\node () at (0.6, 0.4) {$1$};
%            \end{tikzpicture}\right) = 4.
%            \end{align*}
%Let $\Gamma$ be a triangle with vertices $v_1, v_2, v_3$. Then the multiplicity of $D = v_1 + v_2 + v_3$ is $2$, and all other balanced divisors have multiplicity $1$, see Figure \ref{triangle} and formula \eqref{triangleB}.
%%\begin{align*}
%%\mult\left(
%%            \begin{tikzpicture}[baseline={([yshift=-.5ex]current bounding box.center)},vertex/.style={anchor=base,
%%    circle,fill=black!25,minimum size=18pt,inner sep=2pt}]
%%   		 \draw   (0,0) -- (0.5, 0.86);
%%   		 \draw   (0,0) -- (1,0);    
%% 		   \draw    (0.5, 0.86) -- (1,0); 
%%		\node () at (-0.1,-0.1) {$1$};
%%		\node () at (1.1,-0.1) {$1$};	
%%				\node () at (0.5,1.02) {$1$};	
%%            \end{tikzpicture} 
%%\right) = 2,
%%\end{align*}
%\end{example}
\begin{example}
Let $\Gamma$ be a graph with $2$ vertices joined by $n$ edges. Then
$
B_\Gamma = (x_1 + x_2)^n, 
$
so multiplicities are given by binomial coefficients.
\end{example}
\begin{example}\label{Schur}
Let $\Gamma = K_n$ be a complete graph on $n$ vertices. Then
 $$
 B_\Gamma(x_1, \dots, x_n) = \frac{\mathrm V(x_1^2, \dots, x_n^2)}{\mathrm V(x_1, \dots, x_n)},
 $$
i.e. $B_\Gamma$ is the Schur polynomial corresponding to the partition $\lambda = (n-1, n-2, \dots,1, 0)$. Therefore, indegree divisors on $\Gamma$ are weights of the irreducible representation of $\gl_n$ with highest weight $\lambda$, and the multiplicity of any indegree divisor is equal to the  multiplicity of the corresponding weight.
\end{example}
\begin{remark}\label{circuits}\label{multOfVert}
The multiplicity of an indegree divisor $D$ can also be defined as $1$ plus the number of possibly disconnected  oriented circuits in the directed graph $(\Gamma, \mathfrak o)$ where $\mathfrak o$ is any orientation of $\Gamma$ such that $\indeg(\mathfrak o) = D$, and a \textit{circuit} is a cycle allowing repetitions of vertices but not edges.% Indeed, if $\mathfrak o'$ is another orientation of $\Gamma$ such that $\indeg(\mathfrak o') = D$, then there exists a unique possibly disconnected oriented circuit $\gamma$ in the graph $(\Gamma, \mathfrak o)$ such that if we change the orientation of all edges along $\gamma$, we obtain the graph $(\Gamma, \mathfrak o')$. Therefore, the number of oriented circuits is indeed equal to  $\mult(\Gamma, D) - 1$.\par
\par
Note that this definition of multiplicities in terms of oriented circuits implies that for a loopless graph $\mult(\Gamma, D) = 1$ if and only if $D$ is a vertex of the zonotope $Z_\Gamma$ (cf. Remark~\ref{verticesOfZG}). If the graph $\Gamma$ has $k$ loops, then for any $D \in \BDiv(\Gamma)$, we have $\mult(\Gamma, D) \geq 2^k$, and  $\mult(\Gamma, D) = 2^k$ if and only if $D$ is a vertex of the zonotope $Z_\Gamma$.
\end{remark}
Now, let us define the multiplicity of an indegree divisor along an indegree divisor on a subgraph. Let $\Gamma_1  \subset \Gamma$ be a generating subgraph, and let $\pazocal O(\Gamma \setminus \Gamma_1)$ be the set of all possible orientations of those edges of $\Gamma$ which do not belong to $\Gamma_1$. Then, for each such partial orientation $\mathfrak o$, we can define a divisor $\indeg(\mathfrak o)$ by the same formula \eqref{indegDiv}: the value of $\indeg(\mathfrak o)$ at the vertex $v$ is the number of arrows pointing at $v$. 
Note that if $D_1 \in \BDiv(\Gamma_1)$, and $\mathfrak o \in \pazocal O(\Gamma \setminus \Gamma_1)$, then $D_1 + \indeg(\mathfrak o) \in \BDiv(\Gamma)$. 
\begin{definition}\label{relMult}
The \textit{multiplicity} of $D \in \BDiv(\Gamma)$ along $D_1 \in \BDiv(\Gamma_1)$ is the number of orientations $\mathfrak o \in \pazocal O(\Gamma \setminus \Gamma_1)$ such that $ D - D_1 = \indeg(\mathfrak o)$. We denote the multiplicity of $D $ along $D_1 $ as $\mult(\Gamma, D \mid \Gamma_1, D_1)$. 

\end{definition}
% \begin{figure}[t]\centerline{\begin{tikzpicture}[scale = 2]
%   		 \draw  [dashed, ->- ] (0,0) -- (0.5, 0.86);
%		 \draw [dashed, -<- ] (0,0) -- (0.5, 0.3);
%		 \draw [dashed, -<- ] (0.5, 0.3) -- (1,0);
%		  \draw [dashed, -<-] (0.5, 0.3) -- (0.5, 0.86);
%   		 \draw  [] (0,0) -- (1,0);    
% 		   \draw  [dashed, -<-]  (0.5, 0.86) -- (1,0); 
%		\node () at (-0.1,-0.1) {$0$};
%		\node () at (1.1,-0.1) {$1$};	
%				\node () at (0.5,1.02) {$0$};	
%				\node () at (0.6, 0.4) {$0$};
%            \end{tikzpicture} \qquad  \begin{tikzpicture}[scale = 2]
%   		 \draw  [dashed, -<- ] (0,0) -- (0.5, 0.86);
%		 \draw [dashed, ->- ] (0,0) -- (0.5, 0.3);
%		 \draw [dashed, -<- ] (0.5, 0.3) -- (1,0);
%		  \draw [dashed, ->-] (0.5, 0.3) -- (0.5, 0.86);
%   		 \draw  [] (0,0) -- (1,0);    
% 		   \draw  [dashed, -<-]  (0.5, 0.86) -- (1,0); 
%		\node () at (-0.1,-0.1) {$0$};
%		\node () at (1.1,-0.1) {$1$};	
%				\node () at (0.5,1.02) {$0$};	
%				\node () at (0.6, 0.4) {$0$};
%            \end{tikzpicture}}\caption{Partial orientations on the complete graph $K_4$ giving rise to the same balanced divisor.}\label{k4}\end{figure}
\begin{example}\label{mult2example}  
\begin{align*}
   \mult\left( \left.\begin{tikzpicture}[scale = 1.7, baseline={([yshift=-.5ex]current bounding box.center)},vertex/.style={anchor=base,
    circle,fill=black!25,minimum size=18pt,inner sep=2pt}]
   		 \draw  [] (0,0) -- (0.5, 0.86);
		 \draw [] (0,0) -- (0.5, 0.3);
		 \draw (0.5, 0.3) -- (1,0);
		  \draw (0.5, 0.3) -- (0.5, 0.86);
   		 \draw  [] (0,0) -- (1,0);    
 		   \draw  []  (0.5, 0.86) -- (1,0); 
		\node () at (-0.1,-0.) {$1$};
		\node () at (1.1,-0.) {$1$};	
				\node () at (0.5,1.02) {$2$};	
				\node () at (0.6, 0.4) {$2$};
            \end{tikzpicture}\right\rvert
            \begin{tikzpicture}[scale = 1.7, baseline={([yshift=-.5ex]current bounding box.center)},vertex/.style={anchor=base,
    circle,fill=black!25,minimum size=18pt,inner sep=2pt}]
   		 \draw  [dashed] (0,0) -- (0.5, 0.86);
		 \draw [dashed] (0,0) -- (0.5, 0.3);
		 \draw [dashed] (0.5, 0.3) -- (1,0);
		  \draw [dashed] (0.5, 0.3) -- (0.5, 0.86);
   		 \draw  [] (0,0) -- (1,0);    
 		   \draw  [dashed]  (0.5, 0.86) -- (1,0); 
		\node () at (-0.1,-0.) {$0$};
		\node () at (1.1,-0.) {$1$};	
				\node () at (0.5,1.02) {$0$};	
				\node () at (0.6, 0.4) {$0$};
            \end{tikzpicture}
            \right) = 2.
\end{align*}

Corresponding partial orientations are 
\begin{align*}
\mbox{}
  \begin{tikzpicture}[scale = 1.7, baseline={([yshift=-.5ex]current bounding box.center)},vertex/.style={anchor=base,
    circle,fill=black!25,minimum size=18pt,inner sep=2pt}]
   		 \draw  [dashed, ->- ] (0,0) -- (0.5, 0.86);
		 \draw [dashed, -<- ] (0,0) -- (0.5, 0.3);
		 \draw [dashed, -<- ] (0.5, 0.3) -- (1,0);
		  \draw [dashed, -<-] (0.5, 0.3) -- (0.5, 0.86);
   		 \draw  [] (0,0) -- (1,0);    
 		   \draw  [dashed, -<-]  (0.5, 0.86) -- (1,0); 
		\node () at (-0.1,-0.1) {$0$};
		\node () at (1.1,-0.1) {$1$};	
				\node () at (0.5,1.02) {$0$};	
				\node () at (0.6, 0.4) {$0$};
            \end{tikzpicture}
         \mbox{ and }
  \begin{tikzpicture}[scale = 1.7, baseline={([yshift=-.5ex]current bounding box.center)},vertex/.style={anchor=base,
    circle,fill=black!25,minimum size=18pt,inner sep=2pt}]
   		 \draw  [dashed, -<- ] (0,0) -- (0.5, 0.86);
		 \draw [dashed, ->- ] (0,0) -- (0.5, 0.3);
		 \draw [dashed, -<- ] (0.5, 0.3) -- (1,0);
		  \draw [dashed, ->-] (0.5, 0.3) -- (0.5, 0.86);
   		 \draw  [] (0,0) -- (1,0);    
 		   \draw  [dashed, -<-]  (0.5, 0.86) -- (1,0); 
		\node () at (-0.1,-0.1) {$0$};
		\node () at (1.1,-0.1) {$1$};	
				\node () at (0.5,1.02) {$0$};	
				\node () at (0.6, 0.4) {$0$};
            \end{tikzpicture}.
\end{align*}
\end{example}

\begin{remark}
Note that for $\Gamma_1 = \emptyset$ and $D_1 = 0$, where $\emptyset \subset \Gamma$ denotes a generating subgraph with no edges, Definition \ref{relMult} reduces to Definition \ref{multDef}, i.e.
$
\mult(\Gamma, D \mid \emptyset, 0) = \mult(\Gamma, D).
$
\end{remark}
%\begin{remark}
%Another way to define the multiplicity of $D$ along $D_1$ is by means of the formula
%$$
%e^{D_1} \cdot \frac{B_\Gamma(x_1, \dots, x_n)}{B_{\Gamma_1}(x_1, \dots, x_n)} = \!\!\!\sum_{D \,\in \, \BDiv(\Gamma)} \!\!\! \mult(\Gamma, D \mid \Gamma_1, D_1) \cdot e^D,
%$$
%or by counting oriented circuits, as in Remark \ref{circuits}.
%\end{remark}
\begin{remark}\label{counter}
Note that a necessary condition for the multiplicity of $D$ along $D_1$ to be nonzero is  $D \geq D_1$. However, this condition is insufficient. For example,
\begin{align*}
   \mult\left( \left.\begin{tikzpicture}[scale = 1.7, baseline={([yshift=-.5ex]current bounding box.center)},vertex/.style={anchor=base,
    circle,fill=black!25,minimum size=18pt,inner sep=2pt}]
   		 \draw  [] (0,0) -- (0.5, 0.86);
		 \draw [] (0,0) -- (0.5, 0.3);
		 \draw (0.5, 0.3) -- (1,0);
		  \draw (0.5, 0.3) -- (0.5, 0.86);
   		 \draw  [] (0,0) -- (1,0);    
 		   \draw  []  (0.5, 0.86) -- (1,0); 
		\node () at (-0.1,-0.1) {$2$};
		\node () at (1.1,-0.1) {$2$};	
				\node () at (0.5,1.02) {$1$};	
				\node () at (0.6, 0.4) {$1$};
            \end{tikzpicture}\right\rvert
            \begin{tikzpicture}[scale = 1.7, baseline={([yshift=-.5ex]current bounding box.center)},vertex/.style={anchor=base,
    circle,fill=black!25,minimum size=18pt,inner sep=2pt}]
   		 \draw  [] (0,0) -- (0.5, 0.86);
		 \draw [dashed] (0,0) -- (0.5, 0.3);
		 \draw [dashed] (0.5, 0.3) -- (1,0);
		  \draw [dashed] (0.5, 0.3) -- (0.5, 0.86);
   		 \draw  [] (0,0) -- (1,0);    
 		   \draw  []  (0.5, 0.86) -- (1,0); 
		\node () at (-0.1,-0.1) {$0$};
		\node () at (1.1,-0.1) {$2$};	
				\node () at (0.5,1.02) {$1$};	
				\node () at (0.6, 0.4) {$0$};
            \end{tikzpicture}
            \right) = 0.
\end{align*}
\end{remark}
\begin{remark}
Another way to define multiplicities is by introducing a poset structure on the set of indegree divisors on generating subgraphs. Let $\Gamma$ be a graph without loops. Consider the set $\mathrm{IN}(\Gamma)$ of pairs $(\Gamma', D)$ where $\Gamma' \subset \Gamma$ is a generating subgraph, and $D \in \BDiv(\Gamma')$ is an indegree divisor on $\Gamma'$. Then the set $\mathrm{IN}(\Gamma)$ has a natural poset structure. Namely, say that $(\Gamma_1, D_1) > (\Gamma_2, D_2)$ if  $\Gamma_1 \supset \Gamma_2$, and there exists an orientation $\mathfrak o \in \pazocal O(\Gamma_1 \setminus \Gamma_2)$ such that $D_1 - D_2 = \indeg(\mathfrak o)$. Let $\pazocal H(\Gamma)$ be the Hasse diagram of the poset $\mathrm{IN}(\Gamma)$. Then multiplicities can be defined by the formula
$$
\mult(\Gamma_1, D_1 \mid \Gamma_2, D_2) = \frac{\mbox{number of directed paths from $(\Gamma_2, D_2)$ to $(\Gamma_1, D_1)$ in } \pazocal H(\Gamma)}{(|E(\Gamma_1)| - |E(\Gamma_2)|)!}.
$$ 
It is worth mentioning that a closely related poset $\overline{\pazocal{OP}_\Gamma}$ was considered in \cite{caporaso2010torelli}. The elements of $\overline{\pazocal{OP}_\Gamma}$ are pairs $(\Gamma', D')$ where $\Gamma' \subset \Gamma$ is a generating subgraph, and $D'$ is a \textit{completely reducible} indegree divisor on $\Gamma'$ (see Section \ref{secRID}). The ordering is defined in the same way as above, i.e. $\overline{\pazocal{OP}_\Gamma}$ is a subposet of the poset $\mathrm{IN}(\Gamma)$. In \cite{Amini}, it is proved that the poset $\overline{\pazocal{OP}_\Gamma}$ is isomorphic to the facet poset of the Voronoi cell associated with the graph $\Gamma$.
%In other words, $\Gamma_1 \supset \Gamma_2$ if $\mult(\Gamma_1, D_1 \mid \Gamma_2, D_2) > 0$.
\end{remark}
\section{Adjacency of strata}
Now, we would like to understand which pairs of strata $\Str(\Gamma, D)$ are adjacent to each other. First note that  an obvious necessary condition for the closure of the stratum $\Str(\Gamma_1, D_1)$  to contain the stratum $\Str(\Gamma_2, D_2)$ is $ \dim \Str(\Gamma_1, D_1) > \dim \Str(\Gamma_2, D_2)$, or, equivalently, $|E(\Gamma_1)| > |E(\Gamma_2)|$. Moreover, since for any double point $(\lambda, \mu) \in C$ the set $\{L\in \pazocal P^B_C \mid \dim \Ker(P(\lambda) - \mu \cdot \Id) = 2\}$ is closed in $\pazocal P^B_C$, we should in fact have $\Gamma_1 \supset \Gamma_2$. However, the latter condition still does not imply that the closure of  $\Str(\Gamma_1, D_1)$ contains $\Str(\Gamma_2, D_2)$.\par
To formulate a necessary and sufficient condition for a stratum to be contained in the closure of another stratum, let us first discuss in more detail the local structure of the variety $\pazocal P^B_C$ in the neighborhood of a stratum $\Str(\Gamma_2, D_2)$. Theorem \ref{thm2} implies that locally $\pazocal P^B_C$ can be written as%s a product of $p$ nodes and $q$ one-dimensional disks. Let us denote a node by $2x+1$ where $x$ is a one-dimensional stratum $\Complex^*$, and $1$ is a zero-dimensional stratum, i.e., a point. Let us also denote a one-dimensional disk by $x$. Then  $\pazocal P^B_C$ can be locally written as
\begin{align*}
(2\Complex^*\sqcup \mbox{a point})^p\times \Complex^q = \bigsqcup_{i=0}^p 2^i\binom{p}{i} (\Complex^*)^i \times \Complex^q.
\end{align*}
where $p = \codim \Str(\Gamma_2 ,D_2)$, and $q = \dim \Str(\Gamma_2 ,D_2)$.
Thus, the neighborhood of the stratum $\Str(\Gamma_2, D_2)$ in the variety $\pazocal P^B_C$ is locally a union of $3^p$ smooth strata whose dimensions vary from $q = \dim \Str(\Gamma_2, D_2)$ to $p+q = \dim \pazocal P^B_C$. Each of these local strata is an open subset of some global stratum $\Str(\Gamma_1, D_1)$.
\begin{definition}
The \textit{multiplicity} of the stratum $\Str(\Gamma_1, D_1)$ along the stratum $\Str(\Gamma_2, D_2)$ is the number of local strata $S$ in the above stratification of the neighborhood of $\Str(\Gamma_2, D_2)$ such that $S \subset \Str(\Gamma_1, D_1)$.
\end{definition}
%\begin{example}\label{threeLinesMult1}
%In Exampe \ref{threeLines}. The multiplicity of $\Str(\Gamma_C, v_1 + v_2 + v_3)$ along the zero-dimensional stratum is $2$. The multiplicity of other three-dimensional strata along the zero-dimensional stratum is $1$.
%%
%%Then $\Str(\emptyset ,0)$ is one point $\diag(a_1+b_1\lambda, a_2+b_2\lambda, a_3 + b_3\lambda)$. Theorem \ref{thm2} implies that near this point $\pazocal P^J_{C}$ is locally a product of three nodes, i.e. a union of eight three-dimensional disks. Six of these disks are the closures of strata depicted in Figure~\ref{vertComp} corresponding to vertices of the permutohedron $P_3$. Two remaining disks belong to the closure of the stratum \eqref{intStratum} corresponding to the interior point of the permutohedron $P_3$ (note that the closure of the stratum \eqref{intStratum}  has a double point at the origin corresponding to a double point of the curve~\eqref{ratCubic}).
%\end{example}
Clearly, the multiplicity of  $\Str(\Gamma_1, D_1)$ along $\Str(\Gamma_2, D_2)$ does not depend on the choice of a point $P \in \Str(\Gamma_2, D_2)$. In particular, if this multiplicity is positive, then the stratum $\Str(\Gamma_2, D_2)$ lies in the closure of $\Str(\Gamma_1, D_1)$, and if the multiplicity is zero,  then $\Str(\Gamma_2, D_2)$ and the closure of $\Str(\Gamma_1, D_1)$ do not intersect each other. Also note that if the multiplicity of $\Str(\Gamma_1, D_1)$ along $\Str(\Gamma_2, D_2)$ is equal to one, then the closure of $\Str(\Gamma_1, D_1)$ is smooth at points of $\Str(\Gamma_2, D_2)$. Multiplicity bigger than one means that the closure of $\Str(\Gamma_1, D_1)$ is singular along $\Str(\Gamma_2, D_2)$.

\par
The following theorem gives a necessary and sufficient condition for a stratum to be contained in the closure of another stratum, and also allows us to compute multiplicities.
%
%
%We say that $\Str(\Gamma_2, D_2)$ belongs to the boundary of $\Str(\Gamma_1, D_1)$ and write
%$$
%\Str(\Gamma_2, D_2) \subset \partial \Str(\Gamma_1, D_1)
%$$
%if $\Str(\Gamma_2, D_2)$ lies in the closure of $\Str(\Gamma_1, D_1)$ in $\pazocal P^B_C$.
%\par
%
% First note that for any double point $(\lambda, \mu) \in C$ the set $\{L\in \pazocal P^B_C \mid \dim \Ker(P(\lambda) - \mu \cdot \Id) = 2\}$ is closed in $\pazocal P^B_C$. This implies that for any $(\Gamma_1, D_1)$ and $(\Gamma_2, D_2)$ such that the closure of the stratum $\Str(\Gamma_1, D_1)$  contains the stratum $\Str(\Gamma_2, D_2)$, we have $\Gamma_1 \supset \Gamma_2$.  \par
%A similar argument shows that $$\overline{\vphantom{\sum}\Str(\Gamma_1, D_1)} \supset \Str(\Gamma_2, D_2)$$
%implies $D_1 \geq D_2$.
%%However, the condition $\Gamma_1 \supset \Gamma_2$ is not sufficient for the closure of $\Str(\Gamma_1, D_1)$  to contain $\Str(\Gamma_2, D_2)$. A necessary and sufficient condition is given by the following theorem.
%
%%This follows from the fact that the set $\{L\in \pazocal P^B_C \mid \dim \Ker(P(\lambda) - \mu \cdot \Id) = 2\}$ where $(\lambda, \mu) \in C$ is a fixed double point is closed in $\pazocal P^B_C$.
%
%%Indeed, if $P \in \Str(\Gamma_1, D_1)$, then for all nodes $(\lambda, \mu) \in C$ corresponding to edges of $\Gamma_C$ which are not in $\Gamma_1$, we have $\dim \Ker(P(\lambda) - \mu \cdot \Id) = 2$, and, therefore, the same should hold for any $P \in \pazocal P^B_C$ lying in the closure of $\Str(\Gamma_1, D_1)$.
%
\begin{theorem}\label{thm3} Let $\Str(\Gamma_1, D_1)$ and $\Str(\Gamma_2, D_2)$ be two strata of $\pazocal P^B_C$. Then 
\begin{longenum}
\item $\Str(\Gamma_2, D_2)$ lies in the closure of $\Str(\Gamma_1, D_1)$ if and only if $\Gamma_2 \subset \Gamma_1$ and
$
\mult(\Gamma_1, D_1 \mid \Gamma_2, D_2) > 0;
$
\item if $\Gamma_2 \subset \Gamma_1$, then the multiplicity of the stratum $\Str(\Gamma_1, D_1)$ along the stratum $\Str(\Gamma_2, D_2)$ is equal to the multiplicity $\mult(\Gamma_1, D_1 \mid \Gamma_2, D_2)$.
\end{longenum}
\end{theorem}
\begin{example}
Let $C$ be three lines in general position, as in Examples \ref{threeLines} and \ref{threeLinesMult0}. Let also $\Gamma= \Gamma_C$, and let $D = v_1+v_2 +v_3 \in \BDiv(\Gamma)$. We have $\mult(\Gamma, D) = 2$ (see~\eqref{triangleB} and Figure~\ref{triangle}). Therefore, the stratum $\Str(\Gamma, D) $ has multiplicity $2$ along the zero-dimensional stratum $\Str(\emptyset,0) = \diag(a_i+b_i\lambda)$. In other words, $P = \diag(a_i+b_i\lambda)$ is a double point for the closure of $\Str(\Gamma, D) $ (cf. Example~\ref{threeLinesMult0}). Moreover, for any $(\Gamma', D') \neq (\emptyset,0) $, we have $\mult(\Gamma, D \mid \Gamma', D') \leq 1$, which implies that the closure of $\Str(\Gamma, D) $ is smooth at all points except $\diag(a_i+b_i\lambda)$.
\end{example}
\begin{example}
More generally, let $C$ be a union of $n$ straight lines (see Examples~\ref{nLines} and  \ref{Schur}). Then maximal-dimension strata $\Str(\Gamma, D) $ of the isospectral variety $\pazocal P^B_C$ are indexed by weights of the representation of $\gl_n$   with highest weight $\lambda = (n-1, n-2, \dots,1, 0)$, and the multiplicity of a stratum $\Str(\Gamma, D) $ at the point $\diag(a_i+b_i\lambda)$ is equal to the multiplicity of the weight $D$. Also note that, in contrast to the case $n=3$, some of the strata  $\Str(\Gamma, D) $ may have multiplicities  greater than $1$ along strata of positive dimension, see, e.g., Example \ref{mult2example}.
%
%%Str corresponding to vertices of $P_n$ are smooth, while any other component is singular. In particular, the multiplicity of an irreducible component at the point  $\diag(a_i+b_i\lambda)$ is equal to the multiplicity of the corresponding weight of $\rho$. Note that apart from the point $\diag(a_i+b_i\lambda)$, irreducible components of $\pazocal P^B_C$ may have other singularities as well, see, e.g., Example \ref{mult2example}.
%%Note that the singular point $\diag(a_i+b_i\lambda)$ is, generally speaking, not isolated.
%
%%The multiplicity of the component corresponding to a lattice point $D \in P_n$
\end{example}
\begin{remark}
Note that for any $(\Gamma_1, D_1)$ and $(\Gamma_2, D_2)$, one has $\mult(\Gamma_1, D_1 \mid \Gamma_2, D_2) \leq \mult(\Gamma_1, D_1)$. Therefore, the closure of a stratum $\Str(\Gamma_1, D_1) $ is smooth if and only if $\mult(\Gamma_1, D_1) = 1$, or, equivalently, if $\Gamma_1$ has no loops, and $D_1$ is a vertex of the graphical zonotope $Z_{\Gamma_1}$ (see Remark \ref{multOfVert}).
\end{remark}

\begin{remark}
Also note that conditions $\Gamma_2 \subset \Gamma_1$ and $D_2 \leq D_1$ are necessary but not sufficient for the stratum $\Str(\Gamma_2, D_2)$ to belong to the closure of the stratum $\Str(\Gamma_1, D_1)$, see, e.g., Example~\ref{counter}.
\end{remark}
\section{Combinatorial interlude III: Reducible indegree divisors}\label{secRID}

Let $\Gamma$ be a graph. Recall that an orientation $\mathfrak o \in \pazocal O(\Gamma)$ is called \textit{strongly connected} if  any two vertices of the directed graph $(\Gamma, \mathfrak o)$ can be joined by a directed path; an orientation $\mathfrak o \in \pazocal O(\Gamma)$ is called \textit{totally cyclic} if every edge of $(\Gamma, \mathfrak o)$ belongs to a directed cycle. It is well-known that if $\Gamma$ is connected, then an orientation $\mathfrak o \in \pazocal O(\Gamma)$ is totally cyclic if and only if it is strongly connected. On a disconnected graph $\Gamma$, an orientation  $\mathfrak o \in \pazocal O(\Gamma)$  is totally cyclic if and only if its restriction to every connected component of $\Gamma$ is strongly connected.
\begin{proposition}\label{irCond}
Let $D \in \BDiv(\Gamma)$ be an indegree divisor. Then the following conditions are equivalent:
\begin{longenum}
\item there exists no proper non-empty subgraph $\Gamma' \subset \Gamma$ such that the restriction of $D$ to $\Gamma'$ is an indegree divisor;
\item for every proper non-empty subgraph $\Gamma' \subset \Gamma$, the degree of the restriction of $D$ to $\Gamma'$ is strictly bigger than the number of edges of $\Gamma'$;
\item there exists a strongly connected orientation $\mathfrak o \in \pazocal O(\Gamma)$ such that  $D = \indeg(\mathfrak o)$;
\item the graph $\Gamma$ is connected, and $D$ is an interior point\footnote{Note that if $Z_\Gamma$ is just one point, then we regard this point as an interior one.} of the graphical zonotope  $ Z_\Gamma$.
\end{longenum}
%An indegree divisor not satisfying these conditions is called \textit{reducible}.
\end{proposition}
\begin{definition}
Let $D \in \BDiv(\Gamma)$ be an indegree divisor. If $D$ satisfies conditions a)-d) listed in Proposition \ref{irCond}, then $D$ is called \textit{irreducible}. Otherwise, $D$ is called \textit{reducible}.
\end{definition}
%\begin{example}
%Let $\Gamma$ be a triangle with vertices $v_1$,$v_2$,$v_3$. Then the only irreducible balanced divisor on $\Gamma$ is $D = v_1 +v_2 + v_3$ (see Figures~\ref{triangle} and~\ref{permut3}).
%\end{example}
\begin{proposition}\label{crCond}
Let $D \in \BDiv(\Gamma)$ be an indegree divisor. Then the following conditions are equivalent:
\begin{longenum}
\item the restriction of $D$ to every connected component of $\Gamma$ is an irreducible divisor\footnote{Note that a restriction of an indegree divisor to a connected component is automatically an indegree divisor.};
\item there exists a totally cyclic orientation $\mathfrak o \in \pazocal O(\Gamma)$ such that  $D = \indeg(\mathfrak o)$;
\item $D$ is an interior point of the graphical zonotope $Z_\Gamma$.
\end{longenum}
\end{proposition}
\begin{definition}
Let $D \in \BDiv(\Gamma)$ be an indegree divisor. If $D$ satisfies conditions a)-c) listed in Proposition \ref{crCond}, then $D$ is called \textit{completely reducible}. 
\end{definition}
\section{Completely reducible matrix polynomials and compactified Jacobians}
\begin{definition}
A matrix polynomial $P(\lambda)$ is called \textit{reducible} if there exists a proper non-empty subspace $W \subset \Complex^n$ which is an invariant under the action of the matrix $P(\lambda)$ for every $\lambda \in \Complex$. A matrix polynomial $P(\lambda)$ is called \textit{completely reducible} if every such invariant subspace $W \subset \Complex^n$ admits an invariant complement.
\end{definition}
Note that if $P(\lambda)$ is reducible, then so is its spectral curve. However, the converse is in general not true. For example, matrix polynomials which belong to the stratum $\Str(\Gamma_C, v_1 + v_2 + v_3)$ in Example \ref{threeLines} are irreducible, though the corresponding spectral curve is a union of three lines. It turns out that a necessary and sufficient condition for reducibility of $P$ can be given in terms of the associated graph $\Gamma_P$ and the divisor $D_P$:
%Now, let us denote by $(\pazocal P_C^B)_{irr}$ the set
\begin{theorem}\label{thm4}
Let $P \in \Str(\Gamma, D)$ be a matrix polynomial. Then $P$ is reducible if and only if $D$ is a reducible divisor on $\Gamma$. Similarly, $P$ is completely reducible if and only if $D$ is completely reducible.
\end{theorem}
Let us introduce the set $$ (\pazocal P_C^B)_{\mathrm{cr}} :=  \{P(\lambda) \in  \pazocal P_C^B \mid P(\lambda) \mbox{ completely reducible}\}.$$ By Theorem \ref{thm4}, we have \begin{align}\label{crStrat} (\pazocal P_C^B)_{\mathrm{cr}} = \bigsqcup\nolimits_{\Gamma, D} \Str(\Gamma,D)\end{align} where $\Gamma$ is a generating subgraph of the dual graph $\Gamma_C$, and $D$ is a completely reducible indegree divisor on $\Gamma$. 
\begin{example}
Let $C$ be three lines, as in Example \ref{threeLines}. Then            $$
          (\pazocal P_C^B)_{\mathrm{cr}} = \Str\left(\begin{tikzpicture}[scale = 1, baseline={([yshift=-.5ex]current bounding box.center)},vertex/.style={anchor=base,
    circle,fill=black!25,minimum size=18pt,inner sep=2pt}]
   		 \draw   (0,0) -- (0.5, 0.86);
   		 \draw   (0,0) -- (1,0);    
 		   \draw    (0.5, 0.86) -- (1,0); 
		\node () at (-0.15,-0.1) {$1$};
		\node () at (1.15,-0.1) {$1$};	
				\node () at (0.5,1.1) {$1$};
            \end{tikzpicture}
 \right) \,\sqcup\, \Str\left(         \begin{tikzpicture}[scale = 1, baseline={([yshift=-.5ex]current bounding box.center)},vertex/.style={anchor=base,
    circle,fill=black!25,minimum size=18pt,inner sep=2pt}]
   		 \draw [dashed]  (0,0) -- (0.5, 0.86);
   		 \draw [dashed]  (0,0) -- (1,0);    
 		   \draw  [dashed]  (0.5, 0.86) -- (1,0); 
		\node () at (-0.15,-0.1) {$0$};
		\node () at (1.15,-0.1) {$0$};	
				\node () at (0.5,1.1) {$0$};
            \end{tikzpicture}\right),
            $$
            cf. the explicit description in Example \ref{threeLines}.
%Accordingly, the set of completely reducible matrix polynomials 

%Then $\Gamma_C$ is a triangle with vertices $v_1$, $v_2$, $v_3$
%%Then the set of irreducible matrix polynomials $P \in \pazocal P_C^B$ is 
\end{example}
Now, we compare stratification \eqref{crStrat} with the stratification of the \textit{canonical compactified Jacobian} introduced by Alexeev. In  \cite{Alexeev} Alexeev showed that if $C$ is a nodal, possibly reducible curve, then in degree $g-1$, where $g$ is the arithmetic genus of $C$, there exists a canonical way of compactifying the generalized Jacobian. This canonical compactified Jacobian has a stratification that can be written in our terms as
\begin{align}\label{ccj}
\overline{\Jac}_{g-1}(C) =\bigsqcup\nolimits_{\Gamma,D} \Jac(C_\Gamma)
\end{align}
where $\Gamma$ is a generating subgraph of the dual graph $\Gamma_C$, $D$ is a completely reducible indegree divisor on $\Gamma$, and, as above, $C_\Gamma$ is the curve obtained from $C$ by resolving all nodes which correspond to edges not in $\Gamma$. \par
Thus, from the combinatorial point of view, stratifications \eqref{crStrat} and \eqref{ccj} coincide. However, the strata themselves are different: each stratum of  \eqref{crStrat} is biholomorphic to an open subset in the Jacobian $\Jac(C_\Gamma \, / \, \infty)$, while strata of  \eqref{ccj} are Jacobians $\Jac(C_\Gamma)$. To pass from $\Jac(C_\Gamma \, / \, \infty)$ to $\Jac(C_\Gamma)$, we should take the quotient  of $\Str(\Gamma,D)$ with respect to the conjugation action of the centralizer $G^B$ of $B$ (see Theorem \ref{thm1}).
%
%So, the set of strata in stratifications \eqref{crStrat} and \eqref{ccj} coincide. However, the strata themselves are different: each stratum of  \eqref{crStrat} is biholomorphic to an open subset in $\Jac(C_\Gamma \, / \, \infty)$, while strata of  \eqref{ccj} are Jacobians $\Jac(C_\Gamma)$. To pass from $\Jac(C_\Gamma \, / \, \infty)$ to $\Jac(C_\Gamma)$, we should take the quotient  of $\Str(\Gamma,D)$ with respect to the conjugation actions of the centralizer $G^B$ of $B$. 
This observation leads to the following conjecture:
\begin{conjecture}\label{conjCCJ}
Assume that $C$ is a nodal curve satisfying condition \eqref{specCond}. Then the moduli space $$  (\pazocal P_C^B)_{\mathrm{cr}} / G^B =  \{P \in \pazocal P_{m,n} \mid P \mbox{ is completely reducible},  C_P = C  \} \,/\,\GL_n(\Complex)$$ of \textit{completely reducible} $n \times n$ degree $m$ matrix polynomials with spectral curve~$C$, considered up to conjugation by constant matrices, is a variety isomorphic to an open dense subset in the canonical compactified Jacobian of $C$.
\end{conjecture}
\bibliographystyle{plain}
\bibliography{MatPolyV3}
\end{document}